\PassOptionsToPackage{final}{graphicx}
\documentclass[11pt,parskip=no, abstract=no]{scrartcl}

\usepackage[a4paper,
            hmargin={2.5cm,2.5cm},
            vmargin={2.5cm,2.5cm},
            heightrounded,
            marginparwidth=2.2cm,
            marginparsep=0.1cm]{geometry}

\usepackage[colorlinks,
            linkcolor={blue},
            citecolor={blue},
            urlcolor={red},
            breaklinks=true,final]{hyperref}

\usepackage[reqno]{amsmath}
\usepackage{amssymb}
\usepackage{amsthm}
\usepackage{cancel}
\usepackage{tikz-cd}
\usepackage{mathtools}
\usepackage[utf8]{inputenc}
\usepackage{scalerel} %
\usepackage[shortlabels]{enumitem}
\usepackage[square]{natbib} %

\usepackage{pedromacros}

\usepackage{tikz-cd}
\usetikzlibrary{decorations.markings}

\usepackage{centernot}
\usepackage[auth-lg]{authblk}

\DeclareMathOperator\supp{supp}

\usepackage[author=anonymous,nomargin,marginclue,footnote,draft]{fixme}
\FXRegisterAuthor{ls}{als}{LS}
\FXRegisterAuthor{pw}{apw}{PW}
\FXRegisterAuthor{pn}{apn}{PN}
\FXRegisterAuthor{dh}{adh}{DH}
\FXRegisterAuthor{sg}{asg}{SG}

\usepackage{comment}

\DeclareSymbolFont{symbolsA}{U}{txsya}{m}{n}
\DeclareSymbolFont{symbolsC}{U}{txsyc}{m}{n}
\DeclareMathSymbol{\multimapdot}{\mathrel}{symbolsC}{20}
\DeclareMathSymbol{\multimapdotinv}{\mathrel}{symbolsC}{21}
\DeclareMathSymbol{\multimap}{\mathrel}{symbolsA}{40}
\DeclareMathSymbol{\multimapinv}{\mathrel}{symbolsC}{18}

\date{}

\title{A Point-free Perspective on Lax extensions and Predicate liftings}

\author[1]{Sergey Goncharov}
\author[2]{Dirk Hofmann}
\author[1]{Pedro Nora}
\author[1]{Lutz Schröder}
\author[1]{Paul Wild}
\affil[1]{\textit{Friedrich-Alexander-Universität Erlangen-Nürnberg, Erlangen, Germany}}
\affil[ ]{\{sergey.goncharov, pedro.nora, lutz.schroeder, paul.wild\}@fau.de}
\affil[2]{\textit{Center for Research and Development in Mathematics and Applications, University of Aveiro, Portugal}}
\affil[ ]{dirk@ua.pt}

\begin{document}

\maketitle

\begin{abstract}
Lax extensions of set functors play a key role in various
  areas including topology, concurrent systems, and modal logic, while
  predicate liftings provide a generic semantics of modal
  operators. We take a fresh look at the connection between lax
  extensions and predicate liftings from the point of view of
  quantale-enriched relations. Using this perspective, we show in
  particular that various fundamental concepts and results arise
  naturally and their proofs become very elementary.  Ultimately, we
  prove that every lax extension is induced by a class of predicate
  liftings; we discuss several implications of this result.
\end{abstract}

\section{Introduction}

A \emph{lax extension} of a given \(\SET\)-functor~$\ftF$ acts on
relations in a way that is (laxly) compatible on the one hand with the
action of~$\ftF$ on sets and maps, and on the other hand with the
algebraic structure on relations, in particular composition. Lax
extensions are a well-established tool in mathematics and computer
science. For example, they are essential in the theory of
\emph{monoidal topology} to encode the type of a space
\citep{HST14}; and in the theory of coalgebras, they encode the
type of a simulation \citep{HJ04} as well as the type of a
cover modality, in the semantics of Moss-type coalgebraic logics
\citep{MV15}.

A more expressive form of coalgebraic modal logic is based on the notion
of predicate lifting, which allows capturing the standard syntax and
semantics of many forms of modal logic found in the literature in a
uniform fashion \citep{CKP+11}.
The connection between the two approaches to coalgebraic modal logic
is governed by the connection between lax extensions and predicate
liftings.  Special predicate liftings, the so-called \emph{Moss
  liftings}, can be extracted from the Barr
extension~\citep{Lea08,KL09}. This principle has been extended to a
large class of lax extensions~\citep{MV15}, and further to the
quantative setting~\citep{WS20}.
Conversely, lax extensions can be constructed from predicate liftings
using the so-called \emph{Kantorovich extension} \citep{WS20}, even in
quantalic generality~\citep{WS21}. Finally, Moss liftings and the
Kantorovich extension lead to a representation theorem (see
Theorem~\ref{p:60} below) for specific ``fuzzy'' (i.e.~$[0,1]$-valued)
lax extensions \citep{WS20}, which is instrumental in deriving a
quantitative Henessy-Milner-type theorem stating essentially that
behavioural distance on coalgebraic systems coincides with logical
distance under suitable bounding assumptions on the functor~\citep{KonigMikaMichalski18,WildSchroder22,FGH+23}.

Although lax extensions apply to relations, the connection between lax
extensions and predicate liftings has previously been expressed
primarily in the language of functions.  In contrast, the cornerstone
of the present work is the principle that the connection between lax
extensions and predicate liftings is best expressed in the language of
relations. More specifically, we work in the language of
quantale-enriched relations as a way of unifying the work developed in
the classical setting and the emerging work in quantitative settings.

\subsection{Main contributions}

The main contribution of this paper is to show that every \emph{quantale-valued} lax extension of an \emph{arbitrary} \(\SET\)-functor is induced by its \emph{class} of Moss liftings.
This generalizes results in the literature  that rely on specific properties of the unit interval and on cardinality constraints on \(\SET\)-functors \citep{WS20,WS21,MV15}.
By tackling this problem from the point of view of relations instead of functions, we obtain elegant proofs and new insights.
For instance, we introduce the notion of predicate lifting for a lax extension which leads to a simple description of Moss lifting that goes beyond the realm of accessible functors and is independent of functor presentations, which feature centrally in previous approaches \citep{Lea08,MV15,WS20}.

The representation result obtained here explains the importance of the canonical extensions of generalized monotone neighborhood functors in the process of constructing quantale-valued lax extensions (in analogy to the two-valued case~\cite{MV15});
it is a stepping stone to connecting the coalgebraic approaches to behavioural distance via quantale-valued lax extensions and via liftings to categories of quantale-enriched categories, respectively \citep{GHN+23},
and similarly to connecting -- in quantalic generality -- the approaches to coalgebraic logics via lax extensions and via predicate liftings;
 and it helps pave the way to obtaining expressive (monotone) quantale-valued coalgebraic modal logics \citep{GHN+23,FGH+23}.

\subsection{Roadmap}

After briefly reviewing the necessary background in Section~\ref{sec:q-rel+cat}, we show in Section~\ref{sec:lax->pl} how to extract predicate liftings from a lax extension in a canonical way.
This leads to the notion of predicate lifting of a lax extension, which generalizes the notion of Moss lifting.
We characterize the predicate liftings of a lax extension as the ones that obey a Yoneda-type formula involving the lax extension and conclude that the Moss liftings correspond to special representable \(\V\)-functors.

In Sections~\ref{sec:pl->lax} and \ref{sec:small-lax}, we revisit the
Kantorovich extension \citep{WS20,WS21}.  The main technical
contributions of these sections are connected with one of the main
results of \citet{WS20}:
\begin{theorem}
  \label{p:60}
  Every finitarily separable fuzzy lax extension of a \(\SET\)-functor is induced by its set of Moss liftings.
\end{theorem}
\noindent A simpler version of this result states that every fuzzy
lax extension of a \emph{finitary} \(\SET\)-functor is induced by
its set of Moss liftings.  Intuitively, a \emph{finitarily separable}
lax extension is a lax extension under which the functor can be
approximated by its finitary part.

In this regard, we show in Section~\ref{sec:pl->lax} that \emph{every}
lax extension is induced by its \emph{class} of
\emph{infinitary} Moss liftings, and in Section~\ref{sec:small-lax}
that the role of \(\lambda\)-accessibility is to ensure that it is
sufficient to consider a \emph{set} of Moss liftings of arity less
than \(\lambda\). Consequently, we obtain that every lax extension of
an arbitrary functor is an initial lift of canonical extensions of
generalized monotone neighbourhood functors (Corollary~\ref{p:35}), a
generalization of Theorem~\ref{p:60} above -- whose proof was tied to
the unit interval -- to arbitrary (commutative) quantales
(Corollary~\ref{p:33}), as well as a generalization of the following
key result \citep{MV15} to arbitrary (commutative) quantales and
functors (Corollary~\ref{d:cor:1}):
\begin{theorem}
  \label{p:36}
  A finitary \(\SET\)-functor admits a lax extension to \(\REL\) that preserves identities if and only if it admits a separating set of monotone predicate liftings.
\end{theorem}

\section{Preliminaries}
\label{sec:q-rel+cat}

We briefly review the theory of quantale-enriched categories from the point of view of quantale-enriched relations, as well as lax extensions and predicate liftings, and establish some notation and basic results.
We warn the reader that there are several approaches to these topics in the literature, often with particular notations and nomenclature;
we follow the conventions of \citet{HST14}.

\subsection{Quantale-enriched relations and categories}

A \df{quantale}, more precisely a commutative unital quantale, is a complete
lattice $\V$ that carries the structure of a commutative monoid
\((\V,\otimes,k)\) such that for every $u \in \V$ the map $u \otimes - \colon \V
\to \V$ preserves suprema. Therefore, in a quantale every map $u \otimes -
\colon \V \to \V$ has a right adjoint $\hom(u,-) \colon \V \to \V$, which is
characterized by
\[
  u \otimes v \le w \iff v \le \hom(u,w),
\]
for all $v, w \in \V$. A quantale is \df{non-trivial} if the least
element \(\bot\) of \(\V\) does not coincide with the greatest element
\(\top\). Moreover, a quantale is \df{integral} if \(\top\) is the
unit of the monoid operation \(\otimes\) of \(\V\), which we refer to
as tensor or multiplication.

\begin{remark}
  In categorical parlance, a quantale is a commutative monoid in the monoidal category of complete lattices and suprema-preserving maps.
\end{remark}

\begin{examples}
  \label{p:70}
  Quantales are common in mathematics and computer science.
  \begin{enumerate}
  \item Every frame becomes a quantale with \(\otimes = \wedge\) and
    \(k = \top\).
  \item \label{p:72} Every left continuous \(t\)-norm \citep{AFS06} defines a
    quantale on the unit interval equipped with its natural order.  The
      following are some typical examples.
    \begin{enumerate}
      \item The product \(t\)-norm has tensor given by multiplication and
        \[
        	\hom(u,v) =
        	\begin{cases}
        		\min(\frac{v}{u},1) & \text{if $u \neq 0$,} \\
        		1                   & \text{otherwise.}
        	\end{cases}
        \]
        Via the map \([0,\infty]\to[0,1],u \mapsto e^{-u}\), this
        quantale is isomorphic to the quantale $[0,\infty]$ of
        extended non-negative real numbers used by \citet{Law73} to
        define (generalized) metric spaces.
      \item The infimum \(t\)-norm has tensor given by infimum and
        \[
	  \hom(u,v) =
	  \begin{cases}
	  	1 & \text{if $u \le v$,} \\
	  	v & \text{otherwise.}
	  \end{cases}
        \]
      \item \label{p:73} The \L{}ukasiewicz \(t\)-norm has tensor
        given by $u \otimes v = \max(0,u+v-1)$ and
        $\hom(u,v) = \min(1,1-u+v)$.  This quantale is isomorphic to
        the quantale that is used implicitly in the usual treatment of
        $1$-bounded metric spaces.  More concretly, this quantale is
        isomorphic to the quantale based on the unit interval equipped
        with the dual of the natural order and with tensor given by
        truncated addition, $u\oplus v = \min(u+v,1)$; the map
        \([0,1] \to [0,1] \colon u \mapsto 1-u\) provides an
        isomorphism.
    \end{enumerate}
  \item Every commutative monoid \((M, \cdot, e)\) generates a quantale
    structure on \((\ftP M, \bigcup)\), the free quantale on~\(M\). The
    tensor \(\otimes\) on \(\ftP M\) is defined by
    \[
      A \otimes B = \{ a \cdot b \mid a \in A \text{ and } b \in B\},
    \]
    for all \(A,B \subseteq M\). The unit of this multiplication is the set
    \(\{e\}\).
  \end{enumerate}
\end{examples}
For a quantale $\V$ and sets $X$, $Y$, a \df{$\V$-relation} -- an enriched
relation -- from $X$ to $Y$ is a map $X \times Y \to \V$; we let
$X \relto Y$ denote the space of such maps, and in particular write $r\colon X\relto Y$ to indicate that~$r$ is a $\V$-relation from~$X$ to~$Y$. As for ordinary relations, $\V$-relations can be composed via
``matrix multiplication''. That is, for $r \colon X \relto Y$ and
$s \colon Y \relto Z$, the composite $s \cdot r \colon X \relto Z$ is calculated
pointwise by
\[
  (s \cdot r)(x,z)=\bigvee_{y \in Y} r(x,y) \otimes s(y,z),
\]
for $x \in X$, $z \in Z$. The collection of all sets and
\(\V\)-relations between them forms the category \(\Rels{\V}\). For
each set \(X\), the identity morphism on \(X\) is the \(\V\)-relation
\(1_X \colon X \relto X\) that sends every diagonal element to~\(k\)
and all the others to~\(\bot\).

\begin{examples}
  The category of relations enriched in the two-element frame is the usual
  category \(\REL\) of sets and relations.  Relations enriched in
  left continuous \(t\)-norms are often called \emph{fuzzy} or
  \emph{quantitative} relations.
\end{examples}

We can compare \(\V\)-relations of type \(X \relto Y\) using the \df{pointwise order},
\[
  r \leq s \iff \forall (x,y) \in X \times Y,\, r(x,y) \leq s(x,y).
\]
Every hom-set of \(\Rels{\V}\) becomes a complete lattice when equipped with
this order, and an easy calculation shows that \(\V\)-relational composition
preserves suprema in each variable. Therefore, \(\Rels{\V}\) is a quantaloid and
enjoys pleasant properties inherited from \(\V\). In particular, precomposition
and poscomposition with a \(\V\)-relation \(r \colon X \relto Y \) define maps
with right adjoints that compute Kan lifts and extensions, respectively. The
\df{lift of a \(\V\)-relation} \(s \colon Z \relto Y\) along
\(r \colon X \relto Y\) is the \(\V\)-relation
\(r \multimapdot s \colon Z \relto X\) defined by the property
\[
  r \cdot t \leq s \iff t \leq r \multimapdot s,
\]
for every \(t \colon Z \relto X\),
\begin{displaymath}
  \begin{tikzcd}[row sep=huge, column sep=huge]
    Z & {} \\
    X & Y
    \ar[from=1-1, to=2-1, dotted, "r \multimapdot s"]
    \ar[from=1-1, to=2-1, bend right=60, "t"']
    \ar[from=1-1, to=2-1, bend right=30, phantom, "\scaleobj{0.7}{\leq}"']
    \ar[from=1-1, to=2-2, "s"] \ar[from=2-1, to=2-2, "r"']
    \ar[from=2-1, to=1-2, near start, phantom, "\scaleobj{0.7}\leq"].
  \end{tikzcd}
\end{displaymath}
We can compute lifts explicitly as
\[
  (r \multimapdot s) (z,x) = \bigwedge_{y \in Y} \hom(r(x,y), s(z,y)).
\]

Dually, the \df{extension of a \(\V\)-relation} \(s \colon Y \relto Z\) along
\(r \colon Y \relto X\) is the \(\V\)-relation
\(s \multimapdotinv r \colon Y \relto Z\) defined by the property
\[
  t \cdot r \leq s \iff t \leq s \multimapdotinv r,
\]
for every \(t \colon Y \relto Z\),
\begin{displaymath}
  \begin{tikzcd}[row sep=huge, column sep=huge]
    Z & {} \\
    X & Y
    \ar[from=2-1, to=1-1, dotted, "s \multimapdotinv r"']
    \ar[from=2-1, to=1-1, bend left=60, "t"]
    \ar[from=2-1, to=1-1, bend left=30, phantom, "\scaleobj{0.7}{\leq}"]
    \ar[from=2-2, to=1-1, "s"'] \ar[from=2-2, to=2-1, "r"]
    \ar[from=2-1, to=1-2, near start, phantom, "\scaleobj{0.7}\leq"].
  \end{tikzcd}
\end{displaymath}
Elementwise, we obtain
\[
  (s \multimapdotinv r) (x,z) = \bigwedge_{y \in Y} \hom(r(y,x), s(y,z)).
\]

Being a quantaloid, \(\Rels{\V}\) is a \emph{2-category} and, therefore, it
makes sense to talk about adjoint \(\V\)-relations; as usual,
\(l \colon X \relto Y\) is left adjoint to \(r \colon Y \relto X\), written
\(l \dashv r\), if \(l \cdot r \leq 1_Y\) and \(1_X \leq r \cdot l \).

\begin{proposition}
  \label{p:2}
  Consider \(\V\)-relations \(p \colon W \relto Y\), \(q \colon V \relto Y\),
  \(r, r' \colon X \relto Y\) and \(s \colon Y \relto Z\).
  \begin{enumerate}
  \item \label{p:12}
    \(r \leq r' \implies q \multimapdot r \leq q \multimapdot r'\).
  \item \label{p:4}
    \(r \leq r' \implies r' \multimapdot q \leq r \multimapdot q\).
  \item \label{p:8}
    \((p \multimapdot q) \cdot (q \multimapdot r) \leq p \multimapdot r\)
  \item \label{p:14}
    \(q \multimapdot r \leq (s \cdot q) \multimapdot (s \cdot r)\).
  \end{enumerate}
\end{proposition}
\begin{proof}
  All claims follow straightforwardly from the universal properties of lifts and
  extensions.
\end{proof}

If \(\V\) is non-trivial, we can see \(\Rels{\V}\) as an extension of \(\SET\)
through the faithful functor \({(-)}_\circ \colon \SET \to \Rels{\V}\) that acts
as identity on objects and interprets a function \(f \colon X \to Y\) as the
\(\V\)-relation \(f_\circ \colon X \relto Y\) that sends every element of the
graph of \(f\) to \(k\) and all the others to \(\bot\). To avoid unnecessary use
of subscripts usually we write \(f\) instead of \(f_\circ\).

The canonical isomorphism \(X \times Y \simeq Y \times X\) in \(\SET\) induces a
contravariant involution in \(\Rels{\V}\)
\[
  {(-)}^\circ \colon {\Rels{\V}}^\op \longrightarrow \Rels{\V}
\]
that maps objects identically and sends a \(\V\)-relation \(r \colon X \relto Y\)
to the \(\V\)-relation \(r^\circ \colon Y \relto X\) defined by
\(r^\circ(y,x) = r(x,y)\), the \df{converse} of \(r\).

\begin{remark}
  \label{p:40}
  The converse of a function \(f \colon X \to Y\) yields an adjunction
  \(f \dashv f^\circ\) in \(\Rels{\V}\) (and this property of
  functions distinguishes them among \(\V\)-relations precisely when
  the quantale \(\V\) is integral and lean
  \citep[Proposition~III.1.2.1]{HST14}). For every set \(Z\), this
  adjunction extends to the following ones:
  \begin{enumerate}
    \item \label{p:41} \( f \cdot (-) \dashv f^\circ \cdot (-) \colon
      \Rels{\V}(Z,X) \to \Rels{\V}(Z,Y)\);
    \item \((-) \cdot f^\circ \dashv (-) \cdot f \colon
      \Rels{\V}(X,Z) \to \Rels{\V}(Y,Z)\).
  \end{enumerate}
\end{remark}

The next results collect some useful facts about the interplay
between extensions, functions, relations, and involution.

\begin{proposition}
  \label{p:5}
  Let \(r \colon X \relto Z\) and \(s \colon Y \relto Z\) be
  \(\V\)-relations, and let \(f \colon A \to X\) and
  \(g \colon B \to Y\) be functions. Then the following holds.
  \begin{enumerate}
    \item \label{p:39}
      \({(s \multimapdot r)}^\circ = r^\circ \multimapdotinv s^\circ\);
    \item \label{p:42} \(g^{\circ}\cdot(s\blackright r)=(s\cdot g)\blackright r\);
    \item \label{p:67} \((s\blackright r)\cdot f=s\blackright(r\cdot f)\).
  \end{enumerate}
\end{proposition}
\begin{proof}
  All claims are consequences of the uniqueness of adjoints. For example, to
  show \ref{p:42} it is sufficient to observe that from Remark~\ref{p:40},
  it follows that both \((s\cdot g)\blackright-\) and
  \(g^{\circ}\cdot(s\blackright -)\) are right adjoint to \((s\cdot g)\cdot-\).
\end{proof}

\begin{proposition}
  \label{p:28}
  Let \(r \colon X \relto A\) and \(s \colon Y \relto A\) be \(\V\)-relations.
  Then
  \[
    s \multimapdot r = \bigwedge_{a \in A} (a^\circ \cdot s) \multimapdot
    (a^\circ \cdot r),
  \]
  where \(a^\circ\) denotes the converse of the map \(a \colon 1 \to A\) that
  selects the element \(a\).
\end{proposition}

\begin{corollary}
  Let \(r \colon X \relto Y\), \(s \colon Y \relto Y\) be \(\V\)-relations, and
  \(y \colon 1 \to Y\) a function. Then
  \[
    y^\circ \cdot (s \multimapdot r) = y^\circ \cdot (y^\circ \cdot s)
    \multimapdot (y^\circ \cdot r).
  \]
\end{corollary}
\begin{proposition}
  \label{p:10}
  Let \(r \colon X \relto Y\) and \(s \colon A \relto B \) be \(\V\)-relations,
  and let \(f \colon X \to A\) and \(g \colon Y \to B\) be functions such that
  \(g \cdot r \leq s \cdot f \). Then, \(r \cdot f^\circ \leq g^\circ \cdot s\).
\end{proposition}
\begin{proof}
  \(r \cdot f^\circ \leq g^\circ \cdot g \cdot r \cdot f^\circ \leq g^\circ
  \cdot s \cdot f \cdot f^\circ \leq g^\circ \cdot s\).
\end{proof}

Category theory underlines preordered sets as the fundamental ordered
structures.  For an arbitrary quantale \(\V\), the same role is taken
by \(\V\)-categories. Analogously to the classical case, we say that a
\(\V\)-relation \(r \colon X \relto X\) is \df{reflexive} if
\(1_X \leq r\), and \df{transitive} if \( r \cdot r \leq r\). A
\df{$\V$-category} is a pair $(X,a)$ consisting of a set \(X\) of
objects and a reflexive and transitive \(\V\)-relation
\(a \colon X \relto X\); a \df{\(\V\)-functor} \((X,a) \to (Y,b)\) is
map \(f \colon X \to Y \) such that \(f \cdot a \leq b \cdot
f\). Clearly, \(\V\)-categories and \(\V\)-functors define a category,
denoted as \(\Cats{\V}\).

\begin{remark}
  As the nomenclature suggests, the notions of \(\V\)-category and
  \(\V\)-functor come from enriched category theory \citep{Kel82, Law73,
    Stu14}. In fact, unravelling the definition of reflexive and transitive
  \(\V\)-relation yields the typical definition of quantale-enriched category: a
  pair \((X,a)\) consisting of a set $X$ and a map
  $a \colon X \times X \to \V$ that satisfies the
  inequalities
  \begin{align*}
    k \le a(x,x) && \text{and} && a(x,y) \otimes a(y,z) \leq a(x,z)
  \end{align*}
  for all $x,y,z \in X$. Similarly, a $\V$-functor $f \colon (X,a) \to (Y,b)$ is a map
  $f \colon X \to Y$ such that, for all $x,y \in X$,
  \[
    a(x,y) \le b(f(x),f(y)).
  \]
\end{remark}

\begin{examples}
  The following are some familiar examples of quantale-enriched categories.
  \begin{enumerate}
  \item
      The category $\Cats{\two}$ is equivalent to the category $\ORD$ of preordered
      sets and monotone maps.
    \item  Metric, ultrametric and bounded metric spaces \`a la
      \citet{Law73} can be seen as categories enriched in left continuous
      \(t\)-norms:
      \begin{enumerate}
        \item With multiplication \(*\), \(\Cats{[0,1]_*}\) is equivalent to the
          category \(\MET\) of (generalized) \df{metric spaces} and non-expansive
          maps.
        \item With infimum \(\wedge\), \(\Cats{[0,1]_\wedge}\) is equivalent to
          the category $\UMET$ of (generalized) \df{ultrametric spaces} and
          non-expansive maps.
        \item With the \df{\L{}ukasiewicz tensor} \(\luk\),
          $\Cats{[0,1]_{\luk}}$ is equivalent to the category $\BMET$ of
          (generalized) \df{bounded-by-$1$ metric spaces} and non-expansive
          maps.
      \end{enumerate}
    \item Categories enriched in a free quantale \(\ftP M\) on a
      monoid \(M\) (such as $M=\Sigma^*$ for some alphabet~$\Sigma$)
      can be interpreted as labelled transition systems with labels
      in~$M$: in a \(\ftP M\)-category \((X,a)\), the objects represent
      the states of the system, and we can read \(m\in a(x,y)\) as an
      $m$-labelled transition from~\(x\) to~\(y\).
  \end{enumerate}
\end{examples}

\begin{definition}
	The \df{dual} of a \(\V\)-category \((X,a)\) is the \(\V\)-category \({(X,a)}^\op = (X,a^\circ)\).
\end{definition}

\begin{remark}
  \label{p:1}
  The quantale $\V$ becomes a $\V$-category under the canonical structure
  $\hom \colon \V \times \V \to \V$. In fact, for every set $S$, we can form the
  $S$-power $\V^S$ of $\V$ which has as underlying set all functions
  $h \colon S \to \V$ and $\V$-category structure $[-,-]$ given by
  \[
    [h,l] = \bigwedge_{s \in S} \hom(h(s),l(s)),
  \]
  for all $h,l \colon S \to \V$. For instance, for the quantale
  \(([0,1],\oplus,0)\), where~\([0,1]\) is equipped with the dual of the
  natural ordering (Example~\ref{p:70}(\ref{p:73})), this distance on
  \(\Rels{[0,1]}(X,Y)=[0,1]^{X\times Y}\) is given in terms of the
  \emph{natural} order on \([0,1]\) by
  \[
    [r,s] = \sup \big\{\max(s(x,y)-r(x,y),0) \mid (x,y) \in X \times Y\big\}.
  \]
\end{remark}

Every $\V$-category $(X,a)$ carries a natural order defined by
\[
  x \le y \text{ if } k \le a(x,y),
\]
which induces a faithful functor $\Cats{\V} \to \ORD$. A \(\V\)-category
\((X,a)\) is \df{separated} if
\begin{displaymath}
  (k\le a(x,y)\quad\&\quad k\leq a(y,x))\implies x=y
\end{displaymath}
for all \(x,y\in X\).
That is, \((X,a)\) is separated if the natural order
defined above is anti-symmetric.

\begin{remark}\label{d:rem:1}
  The natural order of the $\V$-category $\V$ is just the order of the quantale
  $\V$. The natural order of $\V^S$ is calculated pointwise, and as such is
  complete. Furthermore, the \(\V\)-category~\(\V^{S}\) is complete in the sense
  of enriched category theory; in particular, \(\V^{S}\) has tensors, which are
  given by \((u\otimes h)(s)=u\otimes h(s)\), for \(u\in\V\) and \(h\in\V^{S}\).
  Tensors are compatible with composition, that is,
  \((u \otimes f) \cdot g = u \otimes (f \cdot g)\) for all  \(u \in \V\),
  \(f \colon S \to \V\), and \(g \colon S' \to S\).
  Here we recall that a \(\V\)-category \((X,a)\) is tensored if, for
  every \(x\in X\), the \(\V\)-functor \(a(x,-)\colon X\to\V\) has a left
  adjoint \(-\otimes x \colon\V\to X\) in \(\Cats{\V}\).
  We also note that \(\V\)-functors between tensored \(\V\)-categories are characterized by a pleasant property: a map \(f \colon X\to Y\) between tensored \(\V\)-categories \(X\) and \(Y\) is a \(\V\)-functor if and only if \(f\) is monotone and, for all \(u\in\V\) and \(x\in X\), \(u\otimes f(x)\le f(u\otimes x)\) \citep{Stu06}.
\end{remark}

\subsection{Lax extensions}
\label{sec:prel-lax}

A \df{lax extension} of a functor \(\ftF \colon \SET \to \SET\) to \(\Rels{\V}\)
consists of a map
\[
  (r,r' \colon X \relto Y) \longmapsto (\eF r \colon \ftF X \relto \ftF Y)
\]
such that
\begin{conditions}
\item[\nlabel{p:61}{(L1)}] $r \le r' \implies \eF r \le \eF r'$,
  \item[\nlabel{p:26}{(L2)}] $\eF s\cdot\eF r\le\eF (s\cdot r)$,
  \item[\nlabel{p:0}{(L3)}] $\ftF f \le \eF f$ and
    ${(\ftF f)}^\circ\le\eF(f^\circ)$
\end{conditions}
for all \(r \colon X \relto Y\), \(s \colon Y \relto Z \), and \(f \colon X \to
Y\).  The first condition means precisely that \(\eF\) induces a monotone map
\begin{equation}
  \label{d:eq:1}
  \eF_{X,Y} \colon \Rels{\V}(X,Y) \longrightarrow \Rels{\V}(\ftF X, \ftF Y),
\end{equation}
for all sets \(X\) and \(Y\). We say that a lax extension is \df{\(\V\)-enriched} if this map satisfies the stronger condition of being a \(\V\)-functor, for all sets \(X\) and \(Y\) (see Remark~\ref{p:1}):
\begin{conditions}
  \item[\nlabel{p:37}{(L1')}] \([r,r'] \leq [\eF r, \eF r']\)
\end{conditions}
for all \(r,r' \colon X \relto Y\).
A lax extension~$\eF$ is \df{identity-preserving} if \(\eF 1_X = 1_{\ftF X}\) for every set~\(X\).

\begin{remark}
  It is common to refer to various forms of extensions of
  \(\SET\)-functors to \(\REL\) as \emph{relators}, \emph{relational
    liftings}, or \emph{lax relational liftings}.
\end{remark}

\begin{example}
  \label{p:62}
  The prototypical example of a lax extension is the Barr extension to \(\REL\) of
  a functor \(\ftF \colon \SET \to \SET\) that preserves weak pullbacks.  Taking advantage of the
  fact that every relation \(r \colon X \relto Y\) can be described as a span
  \begin{displaymath}
    \begin{tikzcd}
      X & R & Y \ar[from=1-2, to=1-1, "p_1"'] \ar[from=1-2, to=1-3, "p_2"],
    \end{tikzcd}
  \end{displaymath}
  the Barr extension of~\(\ftF\)  sends~\(r\) to the relation \(\ftF p_2 \cdot {\ftF p_1}^\circ\).
\end{example}

\noindent
\citet{KV16} provide a concise survey oriented towards applications of lax extensions in coalgebra and logic that deals mostly with lax extensions to \(\REL\).
Regarding lax extensions to \(\Rels{\V}\), work within the framework of \emph{monoidal topology} \citep{CH04, Sea05, SS08} encompasses a substantial amount of results.

Our main motivation to study lax extensions stems from the fact that
they provide a framework for the coalgebraic treatment of various
notions of quantale-valued (bi)simulation
\citep{Rut98,HJ04,MV15,WS20,WS21,Gav18,GHN+23}. Recall that an
\df{$\ftF$-coalgebra} for a functor \(\ftF \colon \SET \to \SET\) is
a pair \((X,\alpha)\) consisting of a set~$X$ of \df{states} and a
\df{transition map} \(\alpha\colon X\to\ftF X\); such coalgebras are
viewed as \emph{transition systems}, with~$\ftF$ determining the
transition type (e.g.\ if~$\ftF$ is just powerset, then
$\ftF$-coalgebras are relational transition systems in the usual
sense). Now given a lax extension
\(\eF \colon \Rels{\V} \to \Rels{\V}\) of~\(\ftF\), an
\df{\(\eF\)-simulation} between \(\ftF\)-coalgebras \((X,\alpha)\) and \((Y,\beta)\) is a
\(\V\)-relation \(s \colon X \relto Y\) such that
\[
	\beta \cdot s \leq \eF s \cdot \alpha.
\]
If the lax extension \(\eF\) preserves converse, then \(\eF\)-simulations are more suitably called \df{\(\eF\)-bisimulations}.
Since \(\Rels{\V}\) is a quantaloid, there is the largest \(\eF\)-(bi)simulation between two given coalgebras, which is termed \df{\(\eF\)-(bi)similarity}.
In the two-valued case, it has been shown by \cite{MV15} that the notion of \(\eF\)-bisimilarity arising from an identity-preserving lax extension that preserves converse coincides with the standard coalgebraic notion of behavioural equivalence.
On the other hand, the notion of \(\V\)-enriched lax extension has been instrumental in establishing quantitative Hennessy-Milner and van Benthem type  theorems \citep{WS20,WS21}. It has been introduced with \(\V\)-enrichment replaced with \ref{p:61} together with the condition

\begin{conditions}
\item[(L4)] For every set \(X\) and every \(u\in\V\),
    \(u\otimes 1_{\ftF X}\le \eF(u\otimes 1_{X})\),
\end{conditions}
where \(u \otimes 1_A\) denotes the tensor of \(u\) and \(1_A\)
as defined in Remark~\ref{d:rem:1} (and this condition is shown to be
equivalent to a variant of \(\V\)-enrichment where \(\V\) is equipped
with the symmetrized distance).  The next result records the
equivalence of these and other conditions.

\begin{theorem}\label{d:thm:3}
  For a lax extension \(\eF \colon \Rels{\V} \to \Rels{\V}\), the following assertions are equivalent.
  \begin{tfae}
  \item\label{d:item:3} \(\eF\) is \(\V\)-enriched.
  \item\label{d:item:4} For every set \(X\) and every \(u\in\V\),
    \(u\leq[1_{\ftF X},\eF(u\otimes 1_{X})]\).
  \item\label{d:item:5} For every set \(X\) and every \(u\in\V\),
    \(u\otimes 1_{\ftF X}\le \eF(u\otimes 1_{X})\).
  \item\label{d:item:6} For every \(\V\)-relation \(r \colon X \relto Y\) and every \(u\in\V\),
    \(u\otimes\eF r\le\eF(u\otimes r)\).
  \end{tfae}
\end{theorem}
\begin{proof}
  Note that a lax extension satisfies the monotonicity condition
  \ref{p:61}, therefore the statements \ref{d:item:3} and
  \ref{d:item:6} are equivalent by general results about tensored
  \(\V\)-categories recalled in Remark~\ref{d:rem:1}.  Now assume
  \ref{d:item:3}. Then, for every \(u\in\V\),
  \begin{displaymath}
    u\le [1_{X},u\otimes 1_{X}]\le [\eF 1_{X},\eF(u\otimes 1_{X})]\le [1_{\ftF X},\eF(u\otimes 1_{X})].
  \end{displaymath}
  The implication \ref{d:item:4}\(\Rightarrow\)\ref{d:item:5} follows from the
  adjunction \(-\otimes 1_{\ftF X}\dashv[1_{\ftF X},-]\). Finally, assume
  \ref{d:item:5}. Then, for every \(u\in\V\) and \(r \colon X\relto Y\),
  \begin{displaymath}
    \eF(u\otimes r)
    =\eF(u\otimes (r\cdot 1_{X}))
    =\eF(r\cdot (u\otimes 1_{X}))
    \geq \eF r\cdot \eF(u\otimes 1_{X})
    \geq \eF r\cdot (u\otimes 1_{\ftF X})
    =u\otimes \eF r.\qedhere
  \end{displaymath}
\end{proof}

\begin{remark}\label{d:rem:4}
  If \(k=\top\) in \(\V\), then \(u\otimes 1_{X}\le 1_{X}\) and therefore
  \([u\otimes 1_{X},1_{X}]=k=\top\). Hence,
  \begin{displaymath}
    [1_{X},u\otimes 1_{X}]=[1_{X},u\otimes 1_{X}]\wedge [u\otimes 1_{X},1_{X}],
  \end{displaymath}
  and \(\V\)-enrichment of a lax extension can be equivalently expressed using
  the symmetrization of the canonical structure on \(\V\).
\end{remark}

\begin{remark}
  For the quantale \([0,1]\) of Example~\ref{p:70}(\ref{p:73}), \citet{WS20}
  prove the equivalence \ref{d:item:3}\(\Leftrightarrow\)\ref{d:item:5} of
  Theorem~\ref{d:thm:3}, but with non-expansiveness of \eqref{d:eq:1} defined
  with respect to the symmetric Euclidean metric on \([0,1]\) and with
  \(\Delta_{\varepsilon,X}\) denoting \(\varepsilon\otimes 1_{X}\). Since this
  quantale is integral, Remark~\ref{d:rem:4} ensures that this is equivalent to
  considering the asymmetric distance.
\end{remark}

In the remainder of this subsection we collect some useful properties of lax extensions.
First we note that they preserve certain composites of \(\V\)-relations and functions strictly.

\begin{proposition}
  \label{p:18}
  Let \(\eF \colon \Rels{\V} \to \Rels{\V}\) be a lax extension, and let
  \(f \colon X \to Y\), \(g \colon W \to Z\) be functions and
  \(s \colon Y \relto Z\) a \(\V\)-relation. Then,
  \begin{enumerate}
    \item \label{p:16} \(\eF (s\cdot f)=\eF s\cdot\eF f=\eF s\cdot \ftF f\),
    \item \label{p:20}
      \(\eF (g^\circ\cdot s)=
      \eF(g^{\circ})\cdot\eF s={(\ftF g)}^\circ\cdot\eF s\).
  \end{enumerate}
\end{proposition}
\begin{proof}
  See, for example, \citep[Proposition~III.1.4.4]{HST14}.
\end{proof}

\begin{proposition}
  \label{p:15}
  Let \(\eF \colon \Rels{\V} \to \Rels{\V}\) be a lax functor.  Then,
  items~\ref{p:16} and~\ref{p:20} of Proposition~\ref{p:18} are
  equivalent as conditions on~\(\eF\).  If\/ \(\eF\) satisfies them,
  then \(\eF\) is a lax extension of~$\ftF$.
\end{proposition}
\begin{proof}
  See, for example, \citep[Proposition~III.1.4.3]{HST14}.
\end{proof}

A \df{morphism of lax extensions} \(\alpha \colon (\ftG, \eG) \to (\ftF,\eF_\ftF)\) is a natural transformation \(\alpha \colon \ftG \to \ftF\) that is oplax as a transformation \(\alpha \colon \eG_\ftG \to \eF_\ftF\);
that is, for every \(r \colon X \relto Y\),
\[
  \alpha_Y \cdot \eG_\ftG r \leq \eF_\ftF r \cdot \alpha_X.
\]
When no ambiguities arise, we simply write
\(\alpha \colon \eG_\ftG \to \eF_\ftF\).

Disregarding size constraints, we have that lax extensions and their
morphisms form a category which is topological over the category of
endofunctors on \(\SET\) \citep[Remark~3.5]{SS08}.
Given a family
\(\alpha = {(\alpha_i \colon \ftF \to \ftF_i)}_{i \in I}\) of natural
transformations where each \(\ftF_i\) carries a lax extension
\(\eF_i\) to \(\Rels{\V}\), the \df{initial extension} \(\eF_\alpha\)
is defined by
\[
  \eF_\alpha r
  = \bigwedge_{i \in I} \alpha_Y^\circ \cdot \eF_i r \cdot \alpha_X,
\]
for \(r \colon X \relto Y\). In particular, the supremum and infimum over
a \emph{class} of lax extensions of a functor \(\ftF \colon \SET \to \SET\) is a
lax extension of \(\ftF\).

Every lax extension has a dual as introduced by \citet{Sea05}.  The \df{dual lax
extension} \(\eF^\circ \colon \Rels{\V} \to \Rels{\V}\) of a lax extension \(\eF
\colon \Rels{\V} \to \Rels{\V}\) is the lax extension of \(\ftF \colon \SET \to
\SET\) that is defined by the assignment
\[
  r \longmapsto {\eF (r^\circ)}^\circ.
\]
Notably, this means that we can
symmetrize lax extensions. The \df{symmetrization}
\(\seF \colon \Rels{\V} \to \Rels{\V}\) of a lax extension
\(\eF \colon \Rels{\V} \to \Rels{\V}\) is the lax extension obtained as the
meet of~\(\eF\) and~\(\deF\).

Finally, an important application of lax extensions of a functor
\(\ftF \colon \SET \to \SET\) to \(\Rels{\V}\) is to construct
liftings of \(\ftF\) to \(\Cats{\V}\); that is, endofunctors on
\(\Cats{\V}\) that make the following diagram commute, where the
vertical arrows represent the forgetful functor:
\begin{displaymath}
  \begin{tikzcd}
    \Cats{\V} & \Cats{\V} \\
    \SET      & \SET
    \ar[from=1-1, to=1-2, "\eF"]
    \ar[from=1-1, to=2-1]
    \ar[from=1-2, to=2-2]
    \ar[from=2-1, to=2-2, "\ftF"']
  \end{tikzcd}
\end{displaymath}
The lifting \(\eF \colon \Cats{\V} \to \Cats{\V}\) induced by a lax extension
\(\eF \colon \Rels{\V} \to \Rels{\V}\) sends a \(\V\)-category \((X,a)\) to the
\(\V\)-category \((\ftF X, \eF a)\).

\subsection{Predicate liftings}
\label{ssec:pl}

Given a cardinal \(\kappa\), a \(\kappa\)-ary \df{predicate lifting} for a functor \(\ftF \colon \SET \to \SET\) is a natural transformation
\[
  \mu \colon \ftPVk \longrightarrow \ftPV \ftF,
\]
where, for a set \(Y\), \(\contrapow_Y \colon \SET^\op \to \SET\) denotes the functor
\[
  \SET(-,Y) \colon \SET^\op \to \SET;
\]
that is, we can think of elements of $\ftPVk X$ as $\kappa$-indexed
families of $\V$-valued predicates on~$X$. We say that
\(\mu \colon \ftPVk\to \ftPV \ftF\) is \df{monotone} if
for every set \(X\) the map
\[
	\mu_X \colon \SET(X,\V^\kappa) \to \SET(\ftF X,\V)
\]
is monotone w.r.t the pointwise orders induced by the corresponding powers of the complete lattice \(\V\), and \df{\(\V\)-enriched} if this map is actually a \(\V\)-functor w.r.t to the \(\V\)-categorical structures induced by the corresponding powers of the \(\V\)-category \(\V\) (See Remark~\ref{p:1}).
Note that when talking about monotone or \(\V\)-enriched predicate liftings, instead of \(\ftPVk\colon\SET^{\op}\to\SET\) we actually consider the functors \(\ftPVk\colon\SET^{\op}\to\POST\) and \(\ftPVk\colon\SET^{\op}\to\Cats{\V}\), respectively.
In any case, the functors \(\ftPVk\) are part of an adjunction: in general, for a category~\(\catA\) and an object~\(A\) with powers in~\(\catA\), we have
  \begin{displaymath}
    \begin{tikzcd}[column sep=huge] %
      \SET^{\op} %
      \ar[shift left, start anchor=east, end anchor=west, bend left=25,
      ""{name=U,below}]%
      {r}{\contrapow_{A}} %
      & \catA %
      \ar[start anchor=west,end anchor=east,shift left,bend left=25,
      ""{name=D,above}] %
      {l}{\catA(-,A)} \ar[from=U,to=D,"\bot" description]
    \end{tikzcd}.
  \end{displaymath}
  Moreover, for all functors \(\ftG \colon \SET^{\op}\to\catA\) and \(\ftF \colon\SET\to\SET\), the adjunction above induces a bijection between natural transformations \(\ftG\to\contrapow_{A}\ftF\) and natural transformations \(\ftF\to\ftA(\ftG,A)\).
  In particular, a \(\kappa\)-ary predicate lifting \(\mu \colon \ftPVk\to\ftPV\ftF\) corresponds to a natural transformation
  \begin{equation}\label{eq:mu-bar}
    \overline{\mu}\colon\ftF \longrightarrow\catA(\ftPVk,\V).
  \end{equation}
  For \(\catA=\SET\) (respectively \(\catA=\POST\)) and \(\V=\two\), the codomain of~$\overline{\mu}$ is the generalized (monotone) neighbourhood functor \citep{SchroderPattinon10,MV15}.
  A collection \(M\) of predicate liftings is \df{separating} if the cone
  \begin{displaymath}
    (\overline{\mu}_{X}\colon  \ftF X\longrightarrow\SET(\ftPVk X,\V))_{\mu\in M}
  \end{displaymath}
  (with~$\overline\mu$ as per~\eqref{eq:mu-bar}) is mono (i.e.\ jointy injective) for every set \(X\).

\begin{remark}
  Predicate liftings are instrumental in coalgebraic logic \citep{Pat04,Sch08,CKP+11}.
  As a basic example, consider the Kripke semantics of modal logic.
  Coalgebras for the covariant powerset functor \(\ftP \colon \SET \to \SET\) correspond precisely to Kripke frames.
  In coalgebraic modal logic, the Kripke semantics of the modal logic \(K\) is recovered by interpreting the modal operator \(\Diamond\) as the predicate lifting \(\Diamond \colon \contrapow_2 \to \contrapow_2\ftP\) whose \(X\)-component is
  defined by
  \[
    \Diamond_X(A)= \{B \subseteq X \mid A \cap B \neq \varnothing\},
  \]
  or by interpreting the modal operator \(\Box\) as the predicate lifting \(\Box \colon \contrapow_2 \to \contrapow_2\ftP\) whose \(X\)-component is defined by
  \[
    \Box_X(A)= \{B \subseteq X \mid B \subseteq A\}.
  \]
  Both \(\Diamond\) and \(\Box\) are monotone (=\(\two\)-enriched)
  predicate liftings, and both $\{\Diamond\}$ and $\{\Box\}$ are
  separating for \(\ftP\).  Similarly, \(\V\)-valued predicate
  liftings provide the semantical framework for modal operators in
  \(\V\)-valued coalgebraic modal logic
  \citep{WS20,WS21,GHN+23,FGH+23}. In inductive definitions of
  the semantics of (coalgebraic) modal logics over a coalgebra
  $(X,\alpha)$, formulae are typically interpreted as subsets of the
  state set~$X$. If~$\heartsuit$ is the predicate lifting interpreting a
  modality~$\heartsuit$, and~$Y\subseteq X$ is the interpretation of a
  formula~$\phi$, then the formula $\heartsuit\phi$ is interpreted by
  the set $\alpha^{-1}[\heartsuit_X(Y)]$. For instance, the interpretation
  of~$\Diamond\phi$ in a $\ftP$-coalgebra $(X,\alpha)$, again
  with~$Y\subseteq X$ being the interpretation of~$\phi$, consists of
  all $x\in X$ such that~$\alpha(x)\cap Y\neq\emptyset$, i.e.\ of all
  states that have some successor satisfying~$\phi$.
\end{remark}

\begin{remark}
  \label{p:43}
  Generalizing a corresponding observation for the \(\two\)-valued case~\citep{Sch08}, we note that by the Yoneda lemma, a predicate lifting \(\mu \colon \ftPVk \to \ftPV \ftF\) is equivalently given by a morphism of type \(\ftF \V^\kappa \to \V\), the image of the identity map on \(\V^\kappa\) under $\mu_{\V^\kappa}$.
  In particular, for given~$\kappa$, the collection of all $\kappa$-ary predicate liftings is small.
  The \(X\)-component \(\mu_X \colon \ftPVk X \to \ftPV \ftF X \) of the predicate lifting induced by a morphism \(g\colon \ftF \V^\kappa \to \V\) is defined by
  \[
    \mu_X(f)= g \cdot \ftF f.
  \]
\end{remark}

  In the two-valued case, it has been shown that separating sets of finitary predicate liftings for a finitary \(\SET\)-functor give rise to expressive coalgebraic modal logics \citep{Sch08}, and that a finitary \(\SET\)-functor admits a separating set of finitary \emph{monotone} predicate liftings if and only if the functor admits an identity-preserving lax extension to \(\REL\) \citep{MV15}.
  On the other hand, sets of \(\V\)-enriched predicate liftings satisfying a quantitative analogue of separation feature in expressiveness results for quantitative coalgebraic modal logics for behavioural distances \citep{KonigMikaMichalski18,WS20,WS21,WildSchroder22,FGH+23}.
  We conclude this subsection with a characterization of \(\V\)-enriched predicate liftings, which corresponds to Theorem~\ref{d:thm:3} and can be shown similarly.

\begin{theorem}
  \label{p:49}
  For a predicate lifting \(\mu \colon \ftPVk \to \ftPV\ftF\), the following are equivalent.
  \begin{tfae}
  \item \(\mu\) is \(\V\)-enriched.
  \item \(\mu\) is monotone and for every \(u\in\V\),
    \(u\leq[\mu(1_{\V^\kappa}),\mu(u\otimes 1_{\V^\kappa})]\).
  \item \(\mu\) is monotone and for every \(u\in\V\),
    \(u\otimes \mu(1_{\V^\kappa}) \le \mu(u\otimes 1_{\V^\kappa})\).
  \item \(\mu\) is monotone and for every function \(f \colon X \to \V^\kappa\) and every \(u\in\V\),
    \(u\otimes\mu(f) \le \mu(u\otimes f)\).
  \end{tfae}
\end{theorem}

\begin{remark}
  Analogously to Remark~\ref{d:rem:4}, if \(k = \top\), then \(\V\)-enrichment of predicate liftings can be equivalently expressed using the symmetrization of the canonical structure on \(\V\).
\end{remark}

\section{From Lax Extensions to Predicate Liftings}
\label{sec:lax->pl}

We proceed to investigate the relationship between lax extensions and
predicate liftings, using \(\V\)-relations as a natural language to
express their interaction.  We begin by expressing
predicate liftings in terms of \(\V\)-relations, with an view to
constructing predicate liftings from lax extensions.

We recall that the category \(\SET\) is cartesian closed. In particular, this
means that for every set~\(I\), the evaluation map
\(\ev_I \colon \V^I \times I \to \V\), which we think of as a \(\V\)-relation
\(\ev_I \colon \V^I \relto I\) defined by \(\ev_I(f,i) = f(i)\), exhibits
the function space \(\V^I\) as an exponential object.  Therefore, for every
set~\(X\), currying/uncurrying defines an isomorphism
\[
  \Rels{\V}(X,I) \simeq \SET(X,\V^I).
\]
We denote the right-adjunct of a \(\V\)-relation \(r \colon X \relto I\) by
\(r^\sharp \colon X \to \V^I\), and the left-adjunct of a function
\(f \colon X \to \V^I\) by \(f^\flat \colon X \relto I\).

\begin{remark}
  For every \(\V\)-relation \(r \colon X \relto I\), the universal property of
  \(\ev_I \colon \V^I \times I \to \V\) interpreted in \(\Rels{\V}\) implies
  \(
    r = \ev_I \cdot r^\sharp.
  \)
\end{remark}

\begin{lemma}
  \label{p:55}
  Let \(\V\) be a quantale and \(I\) a set. For all functions
  \(f \colon X \to Y\) and \(g \colon Y \to \V^I\),
  \[
    (g \cdot f)^\flat = g^\flat \cdot f_\circ.
  \]
\end{lemma}
\begin{proof}
  Let \(x \in X\) and \(i \in I\). Then, by definition,
  \[
    (g \cdot f)^\flat (x,i)
    = (g \cdot f)(x)(i)
    = g(f(x))(i)
    = g^\flat(f(x),i)
    = (g^\flat \cdot f_\circ) (x,i).\qedhere
  \]
\end{proof}

In the sequel, given a functor \(\ftF \colon \SET \to \SET\) and a cardinal
\(\kappa\), let \(\Rels{\V}_\circ(\ftF -, \kappa) \colon \SET^\op \to \SET\)
denote the composite functor
\begin{displaymath}
  \begin{tikzcd}[row sep=tiny, column sep=large]
   \SET^\op & \SET^\op & \Rels{\V}^\op & \SET
   \ar[from=1-1, to=1-4, bend left]{rrr}{\Rels{\V}_\circ(\ftF -, \kappa)}
   \ar[from=1-1, to=1-2, "\ftF"']
   \ar[from=1-2, to=1-3, "(-)_\circ"']
   \ar[from=1-3, to=1-4, "{\Rels{\V}(-, \kappa)}"'].
  \end{tikzcd}
\end{displaymath}
The above lemma says essentially that the uncurrying bijection
$(-)^\flat$ is natural, so we immediately obtain the following:

\begin{proposition}
  For every set \(I\) and every functor \(\ftF \colon \SET \to \SET\),
  the functor \(\ftPVI \ftF \colon \SET^\op \to \SET\) is isomorphic to the functor
  \(\Rels{\V}_\circ(\ftF -, I) \colon \SET^\op \to \SET\).
\end{proposition}
\noindent Therefore, we mainly view predicate liftings as natural
transformations
\(\Rels{\V}_\circ(-,\kappa) \to \Rels{\V}_\circ(\ftF -,1)\) from now
on, although to keep the notation simple we still write
\(\ftPVk \to \ftPV \ftF\). Explicitly, naturality of a
transformation
\(\mu\colon\Rels{\V}_\circ(-,\kappa) \to \Rels{\V}_\circ(\ftF -,1)\)
means that for $f\colon X\to Y$ and $r\colon Y\relto \kappa$, we have
\begin{equation*}
  \mu_X(r\cdot f) = \mu_Y(r)\cdot \ftF f.
\end{equation*}
Now, it is easy to see that every lax extension induces natural
transformations with the desired domain:

\begin{proposition}
  Let \(\eF \colon \Rels{\V} \to \Rels{\V}\) be a lax functor, and
  \(\ftF \colon \SET \to \SET\) a functor.  Then~\(\eF\) is a lax
  extension of \(\ftF\) if and only if \(\eF\) agrees with \(\ftF\) on
  objects, and for every set \(I\), \(\eF\) induces a natural
  transformation
  \(\eF_{-,I} \colon \Rels{\V}_\circ (-,I) \to \Rels{\V}_\circ(\ftF -, \ftF I)\).
\end{proposition}
\begin{proof}
  By Propositions~\ref{p:18}(\ref{p:16}) and~\ref{p:15}.
\end{proof}
\noindent This description suggests that to construct \(\kappa\)-ary
predicate liftings from a lax extension, we should construct natural
transformations
\(\Rels{\V}_\circ (\ftF -, F\kappa) \to \Rels{\V}_\circ (\ftF -, 1)\).
The easiest way to do so is to select a \(\V\)-relation
\(\fr \colon \ftF \kappa \relto 1\). Then, each component of the
resulting predicate lifting
\(\Rels{\V}_\circ (-,\kappa) \to \Rels{\V}_\circ (\ftF -, 1)\) is
computed as
\[
  f \longmapsto \fr \cdot \eF f.
\]
This motivates our notion of predicate lifting induced by a lax extension,
which, as we shall explain at the end of this section, generalizes the notion of
\emph{Moss lifting} \citep{Lea08,KL09,MV15,WS21}.

\begin{definition}
	A predicate lifting \(\mu \colon \ftPVk \to \ftPV \ftF\) is \df{induced by} a lax extension \(\eF \colon \Rels{\V} \to \Rels{\V}\), or just a \df{predicate lifting for \(\eF\)}, if there exists a \(\V\)-relation \(\fr \colon \ftF \kappa \relto 1\) such that \(\mu(f) = \fr \cdot \eF f \), for every \(\V\)-relation \(f \colon X \relto \kappa\).
	If \(\fr\) is the converse of an element \(\fk \colon 1 \to \ftF \kappa\), then we say that \(\mu\) is a \df{Moss lifting} of \(\eF\), and we emphasize this by using the notation \(\mu^\fk \colon \ftPVk \to \ftPV \ftF\).
\end{definition}

\noindent Immediately from the definition we have:

\begin{proposition}
  \label{p:64}
  Let \(\eF \colon \Rels{\V} \to \Rels{\V}\) be a lax extension. Every predicate
  lifting induced by \(\eF\) is monotone. If\/ \(\eF\) is \(\V\)-enriched, then
  every predicate lifting induced by\/ \(\eF\) is \(\V\)-enriched.
\end{proposition}

\begin{example}
  \label{p:200}
  Consider the lax extension of the covariant powerset functor
  \(\ftP \colon \SET \to \SET\) to \(\REL\) given by
  \[
    B (\widehat{\ftP}r) C \iff \forall c \in C, \exists b \in B,\, b \mathrel{r} c.
  \]
  The unary Moss lifting for \(\widehat{\ftP}\) determined by the element \(1 \in \ftP 1\) is the predicate lifting \(\Diamond \colon \ftPV \to \ftPV\ftP\) whose \(X\)-component is defined by
  \[
    A \longmapsto \{B \subseteq X \mid A \cap B \neq \varnothing\}.
  \]
  The Moss lifting for the dual extension of \(\widehat{\ftP}\) determined by the element \(1 \in \ftP 1\) is the predicate lifting \(\Box \colon \ftPV \to \ftPV\ftP\) whose \(X\)-component is computed as
  \[
    A \longmapsto \{B \subseteq X \mid B \subseteq A\}.
  \]
  On the other hand, the unary Moss lifting for the Barr extension \(\bar{\ftP}\) of \(\ftP\) (the symmetrization of \(\widehat{\ftP}\)) determined by the element \(1 \in \ftP 1\) is the predicate lifting \(\nabla \colon \ftPV \to \ftPV\ftP\) whose \(X\)-component is defined by
  \[
    A \longmapsto \{B \subseteq X \mid B \neq\varnothing \wedge  B \subseteq A\}.
  \]
\end{example}

As mentioned in Remark~\ref{p:43}, by the Yoneda Lemma, \(\kappa\)-ary predicate liftings for a functor are completely determined by their action on the identity function on \(\V^\kappa\) and the action of the functor.
In the relational point of view, this means that \(\kappa\)-ary predicate liftings are completely determined by the image of the evaluation relation \(\ev_\kappa \colon \V^\kappa \relto \kappa\).
In the following, we will see that \(\kappa\)-ary predicate liftings \emph{induced by a lax extension} are completely determined by their action on the identity \(\V\)-relation \(1_\kappa \colon \kappa \to \kappa\) and the action of the lax extension.
This characterization makes it easy to construct, detect and manipulate predicate liftings induced by lax extensions (see Corollary~\ref{cor:induced}, Lemma~\ref{p:29} and Proposition~\ref{p:69}).

\begin{lemma}
  \label{p:52}
  Let \(\mu \colon \ftPVk \to \ftPV \ftF\) be a predicate lifting
  and \(\eF \colon \Rels{\V} \to \Rels{\V}\) a lax extension. Then the following
   are equivalent:
  \begin{tfae}
    \item \label{p:13} \(\mu\) is induced by \(\eF\);
    \item \label{p:31} \(\mu(\ev_\kappa) = \fr \cdot \eF\ev_\kappa\), for some
      \(\fr \colon \ftF \kappa \relto 1\);
    \item \label{p:81} \(\mu(\ev_\kappa) = \mu(1_\kappa) \cdot \eF\ev_\kappa\);
    \item \label{p:30}
    \(
      \mu(\ev_\kappa)
      = (\mu(\ev_\kappa) \multimapdotinv \eF\ev_\kappa) \cdot \eF\ev_\kappa.
    \)
  \end{tfae}
\end{lemma}
\begin{proof}
  Let \(f \colon X \relto \kappa\) be a \(\V\)-relation.
  \begin{description}
    \item[\ref{p:13} \(\Rightarrow\)~\ref{p:31}] By definition.
    \item[\ref{p:31} \(\Rightarrow\)~\ref{p:81}]
      \(
        \mu(\ev_\kappa)
        = \fr \cdot \eF\ev_\kappa
        = \fr \cdot \eF1_\kappa \cdot \eF\ev_\kappa
        = \mu(1_\kappa) \cdot \eF\ev_\kappa.
      \)
    \item[\ref{p:81} \(\Rightarrow\)~\ref{p:30}] By adjointness, the
      assumption implies that
      $\mu(1_\kappa)\le\mu(\ev_\kappa) \multimapdotinv \eF\ev_\kappa$,
      so we obtain
      \( \mu(\ev_\kappa) = \mu(1_\kappa) \cdot \eF\ev_\kappa \leq
      (\mu(\ev_\kappa) \multimapdotinv \eF\ev_\kappa) \cdot
      \eF\ev_\kappa \leq \mu(\ev_\kappa).  \)
    \item[\ref{p:30} \(\Rightarrow\)~\ref{p:13}] Put
      \(\fr = \mu(\ev_\kappa) \multimapdotinv \eF\ev_\kappa\).  Then, since
      \(\mu\) is a natural transformation,
        \begin{align*}
          \mu(f)
          = \mu(\ev_\kappa \cdot f^\sharp)
          & = \mu(\ev_\kappa) \cdot \ftF f^\sharp \\
          & = \fr \cdot \eF\ev_\kappa \cdot \ftF f^\sharp
          = \fr \cdot \eF f
        \end{align*}
  where the last equality is by Proposition~\ref{p:18}.\qedhere
  \end{description}
\end{proof}
\noindent We crystallize the above into a Yoneda-style
characterization of predicate liftings that is analogous to the one
mentioned in Remark~\ref{p:43} but uses the lax extension instead of
the underlying set functor:

\begin{theorem}
  \label{p:75}
  A predicate lifting \(\mu \colon \ftPVk \to \ftPV\ftF\) is induced by a lax extension \(\eF \colon \Rels{\V} \to \Rels{\V}\) if and only if it is defined from $\mu(1_\kappa)$ (in the \(\V\)-relational view) by
  \[
    \mu(f)= \mu(1_\kappa) \cdot \eF(f).
  \]
\end{theorem}
\begin{proof}
  `If' is immediate by Condition~\ref{p:81} in Lemma~\ref{p:52}; we
  prove `only if'. Varying the calculation in the proof of
  \ref{p:30}$\implies$\ref{p:13} in Lemma~\ref{p:52} and using
  condition~\ref{p:81} of the lemma, we have
  \begin{equation*}
    \mu(f)
    = \mu(\ev_\kappa \cdot f^\sharp)
    = \mu(\ev_\kappa) \cdot \ftF f^\sharp \overset{\text{\ref{p:81}}}{=}\mu(1_\kappa) \cdot \eF\ev_\kappa\cdot \ftF f^\sharp=\mu(1_\kappa)\cdot \eF f. \qedhere
  \end{equation*}
\end{proof}

\begin{corollary}\label{cor:induced}
  A predicate lifting \(\mu \colon \ftPVk \to \ftPV\ftF\) is induced by a lax
  extension \(\eF \colon \Rels{\V} \to \Rels{\V}\) if and only if for every
  relation \(r \colon X \relto Y\) and every function \(g \colon \kappa \to Y\),
  \[
    \mu(g^\circ \cdot r) = \mu(g^\circ) \cdot \eF r.
  \]
\end{corollary}
\begin{proof}
  For `if', just apply the assumption to $g=\id_\kappa$, and use
  Theorem~\ref{p:75}. For `only if', apply Theorem~\ref{p:75} to
  $g^\circ \cdot r$, obtaining
  \begin{equation*}
    \mu(g^\circ \cdot r)= \mu(1_\kappa) \cdot \eF(g^\circ \cdot r)=\mu(1_\kappa) \cdot \eF(g^\circ) \cdot \eF r = \mu(g^\circ)\cdot \eF r
  \end{equation*}
  where the second equality is by Proposition~\ref{p:18}.
\end{proof}

\begin{remark}
  Given a \(\kappa\)-ary predicate lifting \(\mu \colon \ftPVk \to \ftPV\ftF\) induced by a lax extension \(\eF \colon \Rels{\V} \to \Rels{\V}\) of \(\ftF\), it follows from Theorem~\ref{p:75} that \(\mu(1_\kappa) \colon \ftF \kappa \relto 1\) is the largest \(\V\)-relation that induces \(\mu\), but it may not be the only \(\V\)-relation with this property if \(\eF\) does not preserve identities.
  For instance, in Example~\ref{p:200} we have seen that \(\Box \colon \contrapow_2 \to \contrapow_2\ftP\) is the Moss lifting given by \(1 \in \ftP 1\) for the dual lax exntesion of \(\widehat{\ftP}\) but \(\Box(1_1) = \top_{2,1}\).
\end{remark}

In Section~\ref{sec:pl->lax} we will see that every lax extension is induced by its class of Moss liftings, and we turn now to the question of characterizing the Moss liftings induced by a lax extension in the usual point of view.

\begin{lemma}
  \label{p:69}
  Let \(\eF \colon \Rels{\V} \to \Rels{\V}\) be a lax extension of a functor \(\ftF \colon \SET \to \SET\),
  and \(\mu \colon \ftPVk \to \ftPV\ftF\) a predicate lifting induced by \(\eF\).
  Let \(h \colon \V^\kappa \relto \V^\kappa\) denote the structure of the \(\V\)-category
  \({(\V^\kappa)}^\op\). Then,
  \begin{equation*}
    \mu(\ev_\kappa) = \mu(1_\kappa) \cdot {(\ftF 1_\kappa^\sharp)}^\circ \cdot \eF h.
  \end{equation*}
\end{lemma}
\begin{proof}
  Observe that \(\ev_\kappa = {(1_\kappa^\sharp)}^\circ \cdot h\).
  Thus, we have
  \begin{equation*}
    \mu(\ev_\kappa)=\mu({(1_\kappa^\sharp)}^\circ \cdot h)=\mu({(1_\kappa^\sharp)}^\circ)\cdot\eF(h)=\mu(1_\kappa) \cdot \eF ((1_\kappa^\sharp)^\circ) \cdot \eF h=\mu(1_\kappa) \cdot {(\ftF 1_\kappa^\sharp)}^\circ \cdot \eF h,
  \end{equation*}
  using Corollary~\ref{cor:induced}, Theorem~\ref{p:75}, and Proposition~\ref{p:18}, in that order.
\end{proof}

\begin{theorem}
  Let \(\kappa\) be a cardinal and \(h \colon \V^\kappa \relto \V^\kappa\) be the structure of the \(\V\)-category \({(\V^\kappa)}^\op\) (see Remark~\ref{p:1}).
  Furthermore, let \(\eF \colon \Rels{\V} \to \Rels{\V}\) be a lax extension of a functor \(\ftF \colon \SET \to \SET\).
  The \(\kappa\)-ary Moss liftings induced by \(\eF\) correspond precisely to the representable \(\V\)-functors \((\ftF\V^\kappa, \eF h) \to \V^\op\) with representing objects in the image of the map \(\ftF1_\kappa^\sharp \colon \ftF\kappa \to \ftF\V^\kappa\).
\end{theorem}
\begin{proof}
  Let \(\mu^\fk \colon \ftPVk \to \ftPV\ftF\) be a Moss lifting induced by \(\eF\).
  Note that \(\mu(1_\kappa) \cdot {(\ftF 1_\kappa^\sharp)}^\circ \cdot \eF h = \fk^\circ \cdot {(\ftF 1_\kappa^\sharp)}^\circ \cdot \eF h\).
  Hence, from Lemma~\ref{p:69}, we conclude that under the isomorphism between \(\kappa\)-ary predicate liftings for \(\ftF\) and maps of type \(\ftF\V^\kappa \to \V\) (see Remark~\ref{p:43}) the predicate lifting \(\mu^\fk\) corresponds to the map \(\ftF\V^\kappa \to \V\) defined by
  \[
    \label{p:201}
    \fv \longmapsto \eF h(\fv,\ftF 1_\kappa^\sharp(\fk)).
  \]
  That is, \(\mu^\fk\) corresponds to the representable \(\V\)-functor \((\ftF\V^\kappa, \eF h) \to \V^\op\) with representing object \(\ftF1_\kappa^\sharp(\fk)\).
  On the other hand, this also makes it clear that every representable \(\V\)-functor \((\ftF\V^\kappa, \eF h) \to \V^\op\) with representing object in the image of the map \(\ftF1_\kappa^\sharp \colon \ftF\kappa \to \ftF\V^\kappa\) corresponds to a Moss lifting.
\end{proof}

By Lemma~\ref{p:52}, the predicate liftings induced by
a lax extension \(\eF \colon \Rels{\V} \to \Rels{\V}\) are determined by the
fixed points of the monotone map
\[
  ( - \multimapdotinv \eF\ev_\kappa) \cdot \eF\ev_\kappa
  \colon \Rels{\V}(\ftF\V^\kappa,1) \longrightarrow \Rels{\V}(\ftF\V^\kappa,1),
\]
which are precisely the \(\V\)-relations that can be factorized as
\(\fr \cdot \eF\ev_\kappa\) for some \(\V\)-relation
\(\fr \colon \ftF\kappa \relto 1\).  Therefore, the least fixed point
is obtained by composing the \(\V\)-relation
\(\eF\ev_\kappa \colon \ftF\V^\kappa \relto \ftF\kappa\) with the
\(\V\)-relation \(\bot_{\ftF\kappa,1} \colon \ftF\kappa \relto 1\),
the constant function into \(\bot\). The corresponding \(\kappa\)-ary predicate
lifting sends a \(\V\)-relation \(f \colon X \relto \kappa\) to the
\(\V\)-relation \(\bot_{\ftF X,1} \colon \ftF X \relto 1\), the constant
function into \(\bot\). This is the smallest predicate lifting with respect to
\df{the pointwise order of \(\kappa\)-ary predicate liftings}, which is defined by
\[
  \mu \leq \mu' \iff \forall f \colon X \relto \kappa,\,\mu(f) \leq \mu'(f).
\]
Similarly, the greatest fixed point is obtained by composing the \(\V\)-relation
\(\eF\ev_\kappa \colon \ftF\V^\kappa \relto \ftF\kappa\) with the
\(\V\)-relation \(\top_{\ftF\kappa,1}\), the constant function into \(\top\).
However, the greatest predicate lifting induced by a lax extension is not
necessarily the greatest predicate lifting.

\begin{example}
  \label{p:71}
  Consider the identity functor \(1_\REL \colon \REL \to \REL\) which
  is a lax extension to \(\REL\) of the identity functor
  \(1_\SET \colon \SET \to \SET\). Then, the greatest unary predicate
  lifting of \(1_\SET\) induced by \(1_\REL\) is the identity natural
  transformation \(\contrapow_2 \to \contrapow_2\), while the greatest unary
  predicate lifting for \(1_\SET\) sends a relation \(X \relto 1\) to
  the greatest relation \(X \relto 1\).
\end{example}

The definition of Moss lifting of a lax extension used in the literature is seemingly different from the one presented here \citep{Lea08,KL09,MV15,WS20}. To conclude this section, we show that both definitions are computed in the same way.

We recall that every accessible \(\SET\)-functor admits a presentation
(e.g.~\citep[Proposition 3.9]{AGT10} and Theorem~\ref{d:thm:2}) as
follows. A \df{\(\lambda\)-ary presentation} of a
\(\lambda\)-accessible functor \(\ftF \colon \SET \to \SET\) consists
of a $\lambda$-ary \emph{signature} \(\Sigma\), that is, a set of
operations of arity less than \(\lambda\), and for each operation
\(\sigma \in \Sigma\) of arity \(\kappa\), a natural transformation
\(\sigma \colon {(-)}^\kappa \to \ftF\) such that, for every
\(X \in \SET\), the cocone
\({(\sigma_X \colon X^\kappa \to \ftF X)}_{\sigma \in \Sigma}\) is
epi. Every functor \(\ftF \colon \SET \to \SET\) has a
\(\lambda\)-accessible subfunctor
\(\ftF_\lambda \colon \SET \to \SET\) \citep{AGT10} that maps a
set \(X\) to the set
\[
  \ftF_\lambda X = \bigcup \{\ftF i[\ftF Y] \mid i \colon Y \to X \text{ is a
    subset inclusion and } |Y| < \lambda\}.
\]

The notion of Moss lifting of a lax extension \(\eF \colon
\Rels{\V} \to \Rels{\V}\) was introduced by \citet{MV15}.  They think of
predicate liftings for a functor \(\ftF \colon \SET \to \SET\) as natural
transformations
\[
  {(\ftPV)}^n \longrightarrow \ftPV \ftF,
\]
where \(n\) is a natural number and \({(\ftPV})^n\) denotes the
\(n\)-fold product of \(\ftPV\). Then, given a lax extension
\(\eF \colon \Rels{\V} \to \Rels{\V}\) and a finitary presentation of
\(\ftF_\omega\) with signature \(\Sigma\), each
\(n\)-ary \(\sigma \in \Sigma\) induces a \(n\)-ary predicate lifting for
\(\ftF\) --- a Moss lifting --- whose \(X\)-component is defined by
the assignment
\[
  (f_1, \dots, f_n )
  \longmapsto (\fx \mapsto \eF({\ev_X}^\circ)(\fx, \sigma_{\ftPV X}(f_1, \dots, f_n)).
\]
Of course, the functors \(\ftPVk\) and \({(\ftPV)}^\kappa\) are isomorphic,
and the next proposition shows that both definitions of Moss lifting agree.

\begin{proposition}
  Let \(\eF \colon \Rels{\V} \to \Rels{\V}\) be a lax extension of a functor
  \(\ftF \colon \SET \to \SET\), \(\Sigma\) the signature of a \(\lambda\)-ary
  presentation of \(\ftF_\lambda\), and \(\kappa\) a cardinal.
  Let \(\sigma \colon {(-)}^\kappa \to \ftF_\lambda \to \ftF\) be the natural
  transformation induced by a \(\kappa\)-ary operation symbol
  \(\sigma \in \Sigma\).
  Then, for every \(\V\)-relation \(f \colon X \relto \kappa\),
  \[
    \mu^{\sigma_\kappa(1_\kappa)} (f)
    = {\sigma_{P_\V X} (\mate{f})}^\circ \cdot \eF({\ev}^\circ),
  \]
  where \(\mate{f}\) represents the \(\V\)-relation \(f\) as a function of
  type \(\kappa \to \ftPV X\).
\end{proposition}
\begin{proof}
  Note that by definition of \({\ev_X}^\circ\),
  \(f = {\mate{f}}^\circ \cdot {\ev_X}^\circ\).
  Hence, \(\eF f = {\ftF \mate{f}}^\circ \cdot \eF(\ev^\circ)\).
  Moreover, \({\mate{f}}^\kappa (1_\kappa) = \mate{f}\) by definition of the
  functor \({(-)}^\kappa \colon \SET \to \SET\); or with
  \(\id_\kappa \colon 1 \to \kappa^\kappa\)
  denoting the function that selects the identity morphism on \(\kappa\),
  \({\mate{f}}^\kappa \cdot \id_\kappa = \mate{f}\).
  Therefore, since \(\sigma \colon {(-)}^\kappa \to \ftF\) is a natural
  transformation,
  \begin{align*}
    \mu^{\sigma_\kappa(1_\kappa)} (f)
    & = {\id_\kappa}^\circ \cdot \sigma_\kappa^\circ \cdot \eF f \\
    & = {\id_\kappa}^\circ \cdot \sigma_\kappa^\circ \cdot {\ftF \mate{f}}^\circ \cdot \eF ({\ev_X}^\circ) \\
    & = {\id_\kappa}^\circ \cdot {{\mate{f}}^\kappa}^\circ \cdot \sigma_{\ftPV X}^\circ \cdot \eF ({\ev_X}^\circ) \\
    & = {\sigma_{\ftPV X} (\mate{f})}^\circ \cdot \eF({\ev_X}^\circ).
    \qedhere
  \end{align*}
\end{proof}
\noindent
Therefore, \(\kappa\)-ary Moss liftings are constructed in a very
simple way.  From the usual point of view, a Moss lifting
\(\mu^\fk \colon \ftPVk \to \ftPV \ftF\) of a lax extension
\(\eF \colon \Rels{\V} \to \Rels{\V}\), for
\(\fk \colon 1 \to \ftF \kappa\), is just the map that sends a function
\(f \colon X \to \V^\kappa\) (with left-adjunct
$f^\flat\colon X\relto\kappa$) to the function
\(\eF f^\flat (-,\fk) \colon \ftF X \to \V\). %

\begin{remark}
  \citet{MV15} consider lax extensions of \(\SET\)-functors that are not necessarily finitary, however, to construct Moss liftings they restrict the functor to its finitary part.
  That is, the element \(\fk\) inducing a Moss lifting $\mu^\fk$ belongs to \(\ftF_\omega \kappa \subseteq \ftF \kappa\).
\end{remark}

In terms of maps of type \(\ftF \V^\kappa \to \V\),
Proposition~\ref{p:69} tells us that \(\mu^\fk\) is determined by the map
\(\mu(1_{\V^\kappa}) \colon \ftF\V^\kappa \to \V\) defined by
\[
  \fx \longmapsto \eF h(\fx,\ftF 1_\kappa^\sharp(\fk)),
\]
where \(h \colon \V^\kappa \relto \V^\kappa\) is the structure of the
\(\V\)-category \({(\V^\kappa)}^\op\) (see Remark~\ref{p:1}).  In other words,
\(\kappa\)-ary Moss liftings correspond precisely to the representable
\(\V\)-functors \((\ftF\V^\kappa, \eF h) \to \V^\op\) with representing objects
in the image of the map \(\ftF1_\kappa^\sharp \colon \ftF\kappa \to
\ftF\V^\kappa\).

\section{From Predicate Liftings to Lax Extensions}
\label{sec:pl->lax}

The main goal of this section is to provide a way to construct lax extensions to
\(\Rels{\V}\) from predicate liftings, and to understand which lax extensions
arise in such way.  The first task has already been discharged in earlier work
\citep{WS20,WS21}.
In the following, we describe the Kantorovich extension of a collection of
monotone predicate liftings \citep{WS20,WS21}, but completely from the point of
view of \(\V\)-relations.
In Theorem~\ref{p:34}, we show that every Kantorovich extension arises as an initial extension w.r.t to canonical extensions of generalized monotone neighborhood functors.

Given a functor \(\ftF \colon \SET \to \SET\) and a
predicate lifting \(\mu \colon \ftPVk \to \ftPV \ftF\), this perspective makes
it intuitive to construct a \(\V\)-relation \(\ftF X \relto \ftF Y\) from a
\(\V\)-relation \(r \colon X \relto Y\):

\begin{displaymath}
  \begin{tikzcd}[row sep=large, column sep=large]
    X & & & \ftF X & \\
    Y & \kappa & & \ftF Y & 1.
    \ar[from=1-1, to=2-2, "\relsymb" marking, "g \cdot r"]
    \ar[from=1-1, to=2-1, "\relsymb" marking, "r"']
    \ar[from=1-1, to=2-4, phantom, "\leadsto"]
    \ar[from=2-1, to=2-2, "\relsymb" marking, "g"']
    \ar[from=1-4, to=2-5, "\relsymb" marking, "\mu(g \cdot r)"]
    \ar[from=1-4, to=2-4, dotted, "\relsymb" marking, "\mu (g) \multimapdot \mu(g \cdot r)"']
    \ar[from=2-4, to=2-5, "\relsymb" marking, "\mu (g)"']
  \end{tikzcd}
\end{displaymath}

\begin{theorem}
  \label{p:17}
  Let \(\mu \colon \ftPVk \to \ftPV \ftF \) be a \(\kappa\)-ary
  predicate lifting. For every \(\V\)-relation \(r \colon X \relto Y\), consider
  the \(\V\)-relation \(\eF^\mu r \colon \ftF X \relto \ftF Y\) given
  by
  \begin{equation}
    \label{p:44}
    \eF^\mu r = \bigwedge_{g \colon Y \relto \kappa} \mu(g)
    \multimapdot \mu (g \cdot r).
  \end{equation}
  If \(\mu\) is monotone, then the assignment \(r \mapsto \eF^\mu r \) defines
  a lax extension \(\eF^\mu \colon \Rels{\V} \to \Rels{\V} \), which is
  \(\V\)-enriched whenever \(\mu\) is \(\V\)-enriched.
\end{theorem}
\begin{proof}
  \begin{description}
    \item[\ref{p:61}/\ref{p:37}] Monotonicity is immediate from the fact that
    \begin{displaymath}
      \eF^{\mu}(-)\colon\Rels{\V}(X,Y) \longrightarrow\Rels{\V}(\ftF X,\ftF Y)
    \end{displaymath}
    is a composite of monotone maps; similarly for $\V$-enrichment.
  \item[\ref{p:26}] Let \(r \colon X \relto Y\) and
    \(s \colon Y \relto Z\) be \(\V\)-relations. We have to show that
    $\eF s\cdot\eF r\le\eF (s\cdot r)$.  Let
    \(h \colon Z \to \kappa\). By definition, we have
    \begin{align*}
      \eF^\mu s &\leq  \mu(h) \multimapdot \mu(h \cdot s), \\
      \eF^\mu r &\leq  \mu(h \cdot s) \multimapdot \mu(h \cdot s \cdot r).
    \end{align*}
    Therefore, by Proposition~\ref{p:2}(\ref{p:8})
    \[
      \eF^\mu s \cdot \eF^\mu r \leq \mu (h) \multimapdot
      \mu(h \cdot s \cdot r),
    \]
    which implies the claim.
  \item[\ref{p:0}] Let $f\colon X\to Y$. First, observe that since
    \(\mu\) is a natural transformation, we have
    $\mu(g)\cdot \ftF f=\mu(g\cdot f)$, and therefore
  \[
    \ftF f \leq \mu(g) \multimapdot \mu(g \cdot f),
  \]
  for all \(g \colon Y \relto \kappa\). Hence,
  $\ftF f \leq \eF^\mu f$.  Second, note that because \(\mu\) is
  monotone and natural, we have
  \[
    \mu(i) \leq \mu(i \cdot f^\circ \cdot f) = \mu(i \cdot f^\circ) \cdot \ftF
    f
  \]
  for all \(i \colon X \relto \kappa\).
  Therefore, by Proposition~\ref{p:10}, $\mu(i)\cdot(\ftF f)^\circ\le \mu(i \cdot f^\circ)$, and hence
  \[
    (\ftF f)^\circ \leq \mu(i) \multimapdot \mu(i \cdot f^\circ).
  \]
  Thus, $(\ftF f^\circ)\le \eF^\mu (f^\circ)$. \qedhere
  \end{description}
\end{proof}

The formula \ref{p:44} of Theorem~\ref{p:17} is entailed by the view of predicate
liftings as natural transformations of type
\({\Rels{\V}}_\circ(-,\kappa) \to {\Rels{\V}}_\circ(\ftF -,1)\).
By applying the involution on \(\Rels{\V}\), we could also think of predicate
liftings as natural transformations
\({\Rels{\V}}^\circ(\kappa,-) \to {\Rels{\V}}^\circ(1, \ftF -)\)
between functors defined according to the schema
\begin{displaymath}
  \begin{tikzcd}[row sep=tiny, column sep=large]
    \SET^\op & \SET^\op & \Rels{\V} & \SET
    \ar[from=1-1, to=1-4, bend left]{rrr}{\Rels{\V}^\circ(I, \ftG -)}
    \ar[from=1-1, to=1-2, "\ftG"']
    \ar[from=1-2, to=1-3, "(-)^\circ"']
    \ar[from=1-3, to=1-4, "{\Rels{\V}(I,-)}"'].
  \end{tikzcd}
\end{displaymath}
This point of view would lead us to the dual extension of \ref{p:44}.

\begin{proposition}
  Let \(\mu \colon \Rels{\V}_\circ(-,\kappa) \to \Rels{\V}_\circ(\ftF -, 1)\) be a predicate
  lifting and
  \(\bar{\mu} \colon \Rels{\V}^\circ(\kappa,-) \to \Rels{\V}^\circ(1,\ftF -)\)
  be the natural transformation defined by
  \[
    r \longmapsto {\mu(r^\circ)}^\circ.
  \]
  Then,
  \[
    (\eF^\mu r^\circ)^\circ
    = \bigwedge_{f \colon \kappa \relto X} \bar{\mu}(r \cdot f) \multimapdotinv \bar{\mu}(f).
  \]
\end{proposition}
\begin{proof}
  Taking into account Proposition~\ref{p:2}(\ref{p:39}) and the fact
  that \((-)^\circ\) preserves infima, we have
  \begin{align*}
    (\eF^\mu r^\circ)^\circ
    & = \bigwedge_{g \colon X \relto \kappa} (\mu(g) \multimapdot \mu(g \cdot r^\circ))^\circ \\
    & = \bigwedge_{g \colon X \relto \kappa} \mu(g \cdot r^\circ)^\circ \multimapdotinv \mu(g)^\circ \\
    & = \bigwedge_{g \colon X \relto \kappa} \bar{\mu}(r \cdot g^\circ )
    \multimapdotinv \bar{\mu}(g^\circ) \\
    & = \bigwedge_{f \colon \kappa \relto X} \bar{\mu}(r \cdot f) \multimapdotinv \bar{\mu}(f).
    \qedhere
  \end{align*}
\end{proof}

\begin{definition}
  Let \(\ftF \colon \SET \to \SET\) be a functor, and \(M\) a \emph{class} of
  monotone predicate liftings. The \df{Kantorovich} lax extension of \(\ftF\)
  with respect to \(M\) is the lax extension
  \[
    \eF^M = \bigwedge_{\mu \in M} \eF^\mu.
  \]
\end{definition}

\begin{examples}
  \label{p:47}
  Let \(\V\) be a quantale.
  \begin{enumerate}
  \item \label{p:90} The identity functor on \(\Rels{\V}\) is the Kantorovich
    extension of the identity functor on \(\SET\) with respect to the identity
    natural transformation \(\ftPV \to \ftPV\).
  \item \label{p:91} The largest extension of a functor
    \(\ftF \colon \SET \to \SET\) to \(\Rels{\V}\) arises as the Kantorovich
    extension of \(\ftF\) with respect to the natural transformation
    \(\top \colon \ftPV \to \ftPV \ftF\) that sends every map to the
    constant map \(\top\); and also as the Kantorovich extension with respect to
    the natural transformation \(\bot \colon \ftPV \to \ftPV \ftF\) that
    sends every map to the constant map \(\bot\).
  \item \label{p:48} For a subquantale \(\W\) of \(\V\) (that is,
    \(\W\) is a submonoid of \(\V\) closed under suprema), it is easy
    to construct a unary predicate lifting of the covariant
    \(\W\)-powerset functor \(\ftP \colon \SET \to \SET\), which, in
    terms of \(\V\)-relations, is defined by
    \begin{align*}
      \ftP X & =\Rels{\W}(X, 1) \subseteq \Rels{\V}(X,1), \\
      \ftP f & =  (-) \cdot f^\circ.
    \end{align*}
    A straightforward calculation shows that ``evaluating'' induces a predicate
    lifting \(\Diamond \colon \ftPV \to \ftPV \ftP\) whose \(X\)-component
    is defined by
    \begin{align*}
      \Diamond_X \colon \Rels{\V}(X,1) & \longrightarrow \Rels{\V}(\ftP X, 1). \\
      \phi & \longmapsto \phi \cdot \ev_{\W,X}
    \end{align*}

    Then, by Proposition~\ref{p:2}(\ref{p:14}) and Corollary~\ref{p:28},
    \[
      \widehat{\ftP}^\Diamond (r) = \ev_{\W,Y} \multimapdot r \cdot \ev_{\W,X},
    \]
    for every \(\V\)-relation \(r \colon X \relto Y\). Therefore, for
    \(\W\)-relations \(\phi \colon X \relto 1\) and \(\psi \colon Y \relto 1\),
    \[
      \widehat{\ftP}^\Diamond r (\phi, \psi)
      = \bigwedge_{y \in Y} \hom(\psi(y), \bigvee_{x \in X} r(x,y) \otimes
      \phi(x)).
    \]
    Furthermore, for \(\W = \two\) this formula simplifies to
    \[
      \widehat{\ftP}^\Diamond r (A,B) = \bigwedge_{b \in B} \bigvee_{a \in A}
      r(a,b).
    \]

    \begin{enumerate}
    \item \label{p:45} For \(\W = \V = \two\) we obtain a generalization of the
      upper half of the Egli-Milner order: for every \(r \colon X \relto Y\),
      and all \(A \in \ftP X\) and \(B \in \ftP Y\),
      \[
        A (\widehat{\ftP}^\Diamond r) B \iff \forall b \in B, \exists a \in A,
        a\, r\, b.
      \]
    \item \label{p:46} For \(\W = 2\) and \(\V\) a left continuous t-norm, we
      obtain a generalization of the upper half of the Hausdorff metric: for
      every \(r \colon X \relto Y\), and all \(A \in \ftP X\) and
      \(B \in \ftP Y\),
      \[
        \widehat{\ftP}^\Diamond r (A,B) = \bigwedge_{b \in B} \bigvee_{a \in A}
        r(a,b).
      \]
    \end{enumerate}

  \item The dual lax extensions of the extensions~\ref{p:47}(\ref{p:45}) and
    \ref{p:47}(\ref{p:46}) are generalizations of the lower half of the
    Egli-Milner order, and of the lower half of the Hausdorff metric,
    respectively. Therefore, the symmetrization of these lax extensions are
    generalizations of the Egli-Milner order and the Hausdorff metric.

  \item \label{d:exs:1} Let us now consider a faithful functor \(\ftII{-}
    \colon\catA\to\POST\) with some \(\catA\)-object \(\V\) over the partially
    ordered set \(\V\) and the functor
    \(\DD{\kappa}=\catA(\ftPVk,\V)\colon\SET\to\SET\). Some typical
    examples are \(\catA=\POST\) and \(\catA=\Cats{\V}\); for instance, for
    \(\catA=\POST\) and \(\V=\two\), we obtain the generalized monotone neighbourhood
    functor \(\DD{\kappa}\colon\SET\to\SET\). There is a canonical predicate lifting
    corresponding to the identity transformation \(1 \colon
    \DD{\kappa}\to\DD{\kappa}\), and we denote the induced extension as
    \(\eDD{\kappa}\). Then, for a \(\V\)-relation \(r \colon X\relto Y\) and for
    \(\Phi \colon \ftPVk X\to\V\) and \(\Psi \colon
    \ftPVk Y\to\V\),
    \begin{align*}
      (\eDD{\kappa}r)(\Phi,\Psi)
      & = \bigwedge_{g \colon Y\relto \kappa}\Psi(g)\blackright \Phi(g\cdot r) \\
      & = \bigwedge_{g \colon Y\relto \kappa}\hom(\Psi(g),\Phi(g\cdot r)).
    \end{align*}
    This extension coincides with the one considered by \citet{SS08} for the
    classical monotone neighbourhood functor.
    In particular, it follows that, for the identity \(1_{X}\colon X\relto X\),
    the \(\V\)-category \((\DD{\kappa} X,\eDD{\kappa} 1_{X})\) is separated.

  \end{enumerate}
\end{examples}

\begin{remark}
  To see that the Kantorovich extension defined by \citet{WS20}
  coincides with the one presented here, note that Theorem~\ref{p:17}
  requires \(\mu\) to be monotone, hence we can define \(\ftF^\mu\)
  with respect to \(\V\)-relations of type \(X \relto \kappa\) with
  the same result, that is,
  \[
    \eF^\mu r
    =\bigwedge_{f \colon X \relto \kappa} \mu(f \multimapdotinv r) \multimapdot \mu (f).
  \]
  Moreover, in the language of \citet{WS20}, a pair \((f,g)\) of
  \(\kappa\)-indexed families of maps of type \(X \to \V\) is
  \emph{$r$-non-expansive} precisely when \(g \leq f \multimapdotinv r\), when
  interpreting \(f\) and \(g\) as \(\V\)-relations.
\end{remark}

The following result explains the distinguished role of the canonical extensions of generalized monotone neighbourhood functors in the process of constructing lax extensions.

\begin{theorem}
  \label{p:34}
  Let \(M\) be a collection of predicate liftings for a functor \(\ftF \colon \SET \to \SET\).
  The Kantorovich extension \(\eF^{M}\) is the initial extension of \(\ftF\) with respect to the cone
  \begin{displaymath}
    (\overline{\mu}\colon\ftF \longrightarrow\DD{\kappa})_{\mu\in M}
  \end{displaymath}
  (with~$\overline\mu$ as per~\eqref{eq:mu-bar}) and the lax
  extensions \(\eDD{\kappa}\) of \(\DD{\kappa}\) described in
  Examples~\ref{p:47}(\ref{d:exs:1}).  Also, note that if all
  predicate liftings in \(M\) are even \(\V\)-enriched, then so is
  \(\eF^{M}\).
\end{theorem}
\begin{proof}
	Let \(\mu \colon \ftPVk \to \ftPV \ftF \) be a \(\kappa\)-ary predicate lifting in \(M\).
	Then, for a \(\V\)-relation \(r \colon X\relto Y\) and \(\fx\in\ftF X\) and \(\fy\in\ftF Y\),
	\begin{align*}
      		(\eF^{\mu})(\fx,\fy)
      		& = \bigwedge_{g \colon Y\relto \kappa}\mu(g)(\fy) \blackright \mu (g \cdot r)(\fx)\\
      		& = \bigwedge_{g \colon Y\relto \kappa}\overline{\mu}(\fy)(g) \blackright \overline{\mu}(\fx)(g \cdot r)\\
      		& = (\eDD{\kappa}r)(\overline{\mu}(\fx),\overline{\mu}(\fy)).
    	\end{align*}
    That is, \(\eF^{\mu}\) is the \emph{initial extension} with respect to
    \(\overline{\mu}\colon\ftF\to\DD{\kappa}\) and the lax extension \(\eDD{\kappa}\) of \(\DD{\kappa}\).
\end{proof}

We proceed to collect some properties of Kantorovich extensions.
We begin by observing that Kantorovich extensions are compatible with initial extensions along a natural transformation.
This property will be particularly useful in Section~\ref{sec:small-lax} to generalize Theorem~\ref{p:60}.

\begin{proposition}\label{d:rem:2}
  Let $\eF_\ftF \colon \Rels{\V} \to \Rels{\V}$ be a lax extension of a functor \(\ftF \colon\SET\to\SET\), and let \(i \colon \ftG\to \ftF\) be a natural transformation.
  Consider the initial lax extension \(\eG_i \colon\Rels{\V}\to\Rels{\V}\) of~$\ftG$ with respect to \(i \colon \ftG\to \ftF\).
  If\/ \(\eF_\ftF\) is Kantorovich w.r.t a class \(M\) of monotone predicate liftings, then \(\eG_i\) is Kantorovich w.r.t to the class of monotone predicate liftings
  \[
    M_i = \{ (\ftPV i) \cdot \mu \mid \mu \in M \}.
  \]
\end{proposition}
\begin{proof}
  Clearly, every predicate lifting in \(M_i\) is monotone.
  Now, let \(r \colon X \relto Y\) be a \(\V\)-relation.
  Then,
  \begin{align*}
    \eG_i r
    & = i^\circ_Y \cdot \eF^M r \cdot i_X \\
    & = \bigwedge_{\mu \in M} \Big( \bigwedge_{g \colon Y \relto \arity(\mu)} i^\circ_Y \cdot (\mu (g) \multimapdot \mu (g \cdot r)) \cdot i_X \Big).
  \end{align*}
  Therefore, by Corollary~\ref{p:5},
  \begin{align*}
    \eG_i r
    & = \bigwedge_{\mu \in M} \Big( \bigwedge_{g \colon Y \relto \arity(\mu)} (\mu (g) \cdot i_Y) \multimapdot (\mu (g \cdot r) \cdot i_X) \Big) \\
    & = \bigwedge_{\mu \in M} \Big(\bigwedge_{g \colon Y \relto \arity(\mu)} ((\ftPV i) \cdot \mu (g)) \multimapdot ((\ftPV i) \cdot \mu (g \cdot r)) \Big) \\
    & = \bigwedge_{\mu \in M_i} \eF^\mu r.
      \qedhere
  \end{align*}
\end{proof}

\begin{example}
  We recall that an endofunctor on \(\SET\) is called taut if it preserves inverse images \citep{Man02}.
  Every taut functor \(\ftF \colon\SET\to\SET\) admits a natural transformation
  \begin{displaymath}
    \supp \colon \ftF \longrightarrow\ftU
  \end{displaymath}
  into the monotone neighbourhood functor with \(X\)-component
  \begin{displaymath}
    \supp_{X}\colon \ftF X \longrightarrow\ftU X,\quad
    \fx \longmapsto \{A\subseteq X\mid \fx\in \ftF A\}.
  \end{displaymath}
  We note that every \(\supp_{X}(\fx)\) is actually a filter \citep{Man02,
  Gum05}. As observed by \citet{SS08}, the op-canonical extension
  \citep{Sea05}  of
  a taut functor  is the initial lift with respect to \(\supp\) of the extension
  \(\eU\) of the monotone neighbourhood functor \(\ftU \colon\SET\to\SET\) of
  Example~\ref{p:47}(\ref{d:exs:1}), that is, the extension induced by the
  predicate lifting \(\mu\) corresponding to the identity transformation \(1
  \colon\ftU\to\ftU\). Hence, by Proposition~\ref{d:rem:2}, the op-canonical
  extension of a taut functor \(\ftF\) is induced by the predicate lifting
  \begin{displaymath}
    \ftPV\supp \cdot \mu,
  \end{displaymath}
  which, by adjunction, corresponds to the natural transformation
  \begin{displaymath}
    \supp \colon\ftF \longrightarrow\ftU.
  \end{displaymath}
\end{example}

The next result entails that we can use the Kantorovich extension to major a lax extension by extracting all predicate liftings induced by the lax extension.
This is the first step towards representing lax extensions by collections of predicate liftings.

\begin{proposition}
  \label{p:80}
  Let \(\eF \colon \Rels{\V} \to \Rels{\V}\) be a lax extension, and \(\mu \colon \ftPVk \to \ftPV \ftF\) a predicate lifting induced by \(\eF\).
  Then, \(\eF \leq \eF^\mu\).
\end{proposition}
\begin{proof}
  Let \(r \colon X \relto Y\) be a \(\V\)-relation.
  Then, by \ref{p:26},
  \begin{align*}
    \eF^\mu (r) & = \bigwedge_{g \colon Y \relto \kappa} (\mu(1_\kappa) \cdot \eF(g)) \multimapdot (\mu(1_\kappa) \cdot \eF(g \cdot r)) \\
    & \geq \bigwedge_{g \colon Y \relto \kappa} (\mu(1_\kappa) \cdot \eF(g)) \multimapdot (\mu(1_\kappa) \cdot \eF(g) \cdot \eF (r)) \\
    & \geq \eF (r).
    \qedhere
  \end{align*}
\end{proof}

\begin{corollary}
  \label{p:27}
  Let \(\eF \colon \Rels{\V} \to \Rels{\V}\) be a lax extension, and \(\mu^\fk \colon \ftPVk \to \ftPV \ftF\) a Moss lifting of \(\eF\).
  Then, \(\eF \leq \eF^{\mu^\fk}\).
\end{corollary}

\noindent Notably, as a consequence of the previous results we can use the Kantorovich extension to detect predicate liftings induced by lax extensions.

\begin{proposition}
  \label{p:110}
  Let \(\eF \colon \Rels{\V} \to \Rels{\V}\) be a lax extension,
  and \(\mu \colon \ftPVk \to \ftPV \ftF\) a predicate lifting induced by
  \(\eF\). Then, the predicate lifting \(\mu\) is induced by \(\eF^\mu\).
\end{proposition}
\begin{proof}
  According to Lemma~\ref{p:52}, it suffices to show
  \(\mu(\ev_\kappa) = \mu(1_\kappa) \cdot \eF^\mu\ev_\kappa\).  First, observe
  that by Proposition~\ref{p:80} and Lemma~\ref{p:52}(\ref{p:81}) we
  have
  \[
    \mu(\ev_\kappa)
    = \mu(1_\kappa) \cdot \eF(\ev_\kappa)
    \leq \mu(1_\kappa) \cdot \eF^\mu(\ev_\kappa).
  \]
  Second, note that by definition of \(\eF^\mu\) we obtain
  \[
    \mu(1_\kappa) \cdot \eF^\mu(\ev_\kappa)
    \leq \mu(1_\kappa) \cdot (\mu(1_\kappa) \multimapdot \mu(\ev_\kappa))
    \leq \mu(\ev_\kappa).\qedhere
  \]
\end{proof}

\begin{example}
  The predicate lifting \(\Diamond \colon \ftPV \to \ftPV\ftP\) of
  Example~\ref{p:47}(\ref{p:48}) is induced by the lax extension
  \(\widehat{\ftP}^\Diamond \colon \Rels{\V} \to \Rels{\V}\). With
  \(k \colon 1 \to \W\) denoting the function that selects the element
  \(k \in \W\), we have
  \[
    \Diamond(f) = k^\circ \cdot \widehat{\ftP}^\Diamond f
  \]
  for every \(\V\)-relation \(f \colon X \relto 1\).
\end{example}

We already know from Proposition~\ref{p:64} that every predicate
lifting induced by a lax extension is monotone.  The next example
shows that the converse statement does not hold.

\begin{example}
  \label{p:32}
  Let \([0,1]\) denote the quantale consisting of the unit interval equipped
  with the usual order and multiplication. Consider the unary monotone
  predicate lifting for the identity functor
  \(\mu \colon \ftP_{[0,1]} \to \ftP_{[0,1]}\) determined by the map
  \(\mu(\ev_1) \colon [0,1] \to [0,1]\) defined by
  \[
    \mu(\ev_1)(v) =
    \begin{cases}
      0 \text{ if } v \leq \frac{1}{2}; \\
      1 \text{ otherwise}.
    \end{cases}
  \]
  Given that \([0,1]\) is an integral quantale and \(\mu\) is a unary
  predicate lifting, in order to be induced by \(\eF^\mu\), \(\mu\)
  would need to satisfy the condition
  \[
    \mu(\ev_1) \leq \eF^\mu(\ev_1);
  \]
  that is, for every \(g \colon 1 \relto 1\),
  \[
    \mu(\ev_1) \leq \mu(g) \multimapdot \mu(g \cdot \ev_1).
  \]
  However, with \(g = \frac{2}{3}\) and \(v = \frac{3}{4}\),
  \[
    \mu(\ev_1)\left(\frac{3}{4}\right) = 1
    \not\leq \mu(\ev_1)\left(\frac{2}{3}\right) \multimapdot
    \mu(\ev_1)\left(\frac{1}{2}\right) = 0.
  \]
\end{example}

Finally, we tackle the problem of recovering a lax extension to \(\Rels{\V}\) as the Kantorovich extension w.r.t. some \emph{class} of predicate liftings.

\begin{definition}
	A lax extension \(\eF \colon \Rels{\V} \to \Rels{\V}\) of a functor \(\ftF \colon \SET \to \SET\) is  \df{induced} by a class of monotone predicate liftings \(\Lambda\) for \(\ftF\) if \(\eF\) is the Kantorovich extension w.r.t. \(\Lambda\).
\end{definition}

\begin{lemma}
  \label{p:19}
  Let \(\eF \colon \Rels{\V} \to \Rels{\V}\) be a lax extension of a functor \(\ftF \colon\SET\to\SET\), \(\kappa\) a cardinal, and \(i \colon Y \to \kappa\) a function.
  For every \(\fy \in \ftF Y\) and \(\fk = \ftF i(\fy)\), and all \(\V\)-relations \(r \colon X \relto Y\) and \(s \colon Z \relto \kappa\),
  \begin{enumerate}
    \item \label{p:21} \(\mu^\fk(s) = \fy^\circ \cdot \eF(i^\circ \cdot s)\);
    \item if  \(i\) is a monomorphism, then
      \begin{enumerate}
        \item \label{p:25} \(\mu^{\fk}(i \cdot r) = \fy^\circ \cdot \eF r\);
        \item
          \(
            \mu^\fk (i) \multimapdot \mu^\fk (i \cdot r)
            \leq \fy^\circ \multimapdot (\fy^\circ \cdot \eF r).
          \)
      \end{enumerate}
  \end{enumerate}
\end{lemma}
\begin{proof}
  Note that \(\fk = \ftF i(\fy)\) means \(\fk = \ftF i \cdot \fy\),
  when considering elements as functions.
  \begin{enumerate}
    \item
    \(
      \mu^\fk (s)
      = \fk^\circ \cdot \eF s
      = \fy^\circ \cdot {\ftF i}^\circ \cdot \eF s
      = \fy^\circ \cdot \eF (i^\circ \cdot s)
    \).
    \item
      \begin{enumerate}
        \item Since \(i\) is a monomorphism, \(i^\circ \cdot i = 1_Y\). Therefore, the claim
          follows by applying \ref{p:21} with \(s = i \cdot r\).
        \item Applying~\ref{p:25}, and recalling Proposition~\ref{p:2}(\ref{p:4})
          and Condition~\ref{p:0}, yields
          \begin{align*}
            \mu^\fk (i) \multimapdot \mu^\fk (i \cdot r)
            & = (\fy^\circ \cdot \eF 1_Y) \multimapdot (\fy^\circ \cdot \eF r) \\
            & \leq \fy^\circ \multimapdot (\fy^\circ \cdot \eF r). \qedhere
          \end{align*}
      \end{enumerate}
  \end{enumerate}
\end{proof}

\begin{corollary}
  \label{p:23}
  Let \(\eF \colon \Rels{\V} \to \Rels{\V}\) be a lax extension of a functor \(\ftF \colon \SET \to \SET\),
  \(i\colon \lambda \to \kappa\) a function between cardinals,
  \(\fl\) an element of \(\ftF \lambda\), and \(\fk = \ftF i (\fl)\).
  \begin{enumerate}
    \item \label{p:24} \(\eF^{\mu^\fl} \leq \eF^{\mu^\fk}\).
    \item If \(i\) is mono, then \(\eF^{\mu^\fl} = \eF^{\mu^\fk}\).
  \end{enumerate}
\end{corollary}

Lemma~\ref{p:19} allows approximating a lax extension with respect to the cardinality of the codomain of \(\V\)-relations.

\begin{corollary}
  \label{p:22}
  Let \(\eF \colon \Rels{\V} \to \Rels{\V}\) be a lax extension of a functor \(\ftF \colon \SET \to \SET\),
  \(\kappa\) a cardinal, and \(Y\) a set such that \(|Y| \leq \kappa\).
  Consider the set \(M = \{ \mu^\fk \mid \fk \in \ftF \kappa \}\).
  Then, for every \(\V\)-relation \(r \colon X \relto Y\),
  \(\eF r = \eF^M r\).
\end{corollary}

\noindent Therefore, as the collection of all Moss liftings of a lax extension is not larger than the class of all sets,

\begin{theorem}
  \label{d:thm:1}
  Every lax extension of a \(\SET\)-functor to \(\Rels{\V}\) is induced by its class of Moss liftings.
\end{theorem}

From Theorem~\ref{p:34} we obtain

\begin{corollary}
  \label{p:35}
  Every lax extension of a \(\SET\)-functor to \(\Rels{\V}\) is an initial extension with respect to the lax extensions of the functors \(\DD{\kappa} \colon \SET \to \SET\) (See Theorem~\ref{p:34}).
\end{corollary}

From Proposition~\ref{p:64}, we obtain

\begin{corollary}
  Every \(\V\)-enriched lax extension of a \(\SET\)-functor to \(\Rels{\V}\) is induced by a \emph{class} of \(\V\)-enriched predicate liftings.
\end{corollary}

The next result is a generalization of Theorem~\ref{p:36} mentioned in the introduction.
Note that, for \(\V=\two\) and a lax extension \(\eF\colon \Rels{\V}\to \Rels{\V}\) that preserves converses, \(\eF 1_{X}\) is an equivalence relation on \(X\).
Hence, \(\eF\) is identity-preserving if and only if the ordered set \((\ftF X,\eF1_{X})\) is anti-symmetric.

\begin{corollary}
  \label{d:cor:1}
  A functor \(\ftF \colon \SET \to \SET\) has a separating class of \(\V\)-valued monotone predicate liftings if and only if there is a lax extension \(\eF\colon \Rels{\V}\to \Rels{\V}\) such that, for all sets \(X\), the \(\V\)-category \((\ftF X,\eF1_{X})\) is separated.
\end{corollary}
\begin{proof}
  Suppose that there is such a lax extension
  \(\eF\colon \Rels{\V}\to \Rels{\V}\). Let \(X\) be a set and
  \(\fx,\fy\in\ftF X\) with \(\fx\neq\fy\). By assumption, then have
  w.l.o.g.\ that \(k\nleq\eF 1_{X}(\fx,\fy)\). Hence, by
  Theorem~\ref{d:thm:1}, there is a predicate lifting
  \(\mu \colon \ftPVk\to\ftPV\ftF\) such that
  \begin{displaymath}
    k\nleq\DD{\kappa}(\overline{\mu}(\fx),\overline{\mu}(\fy)),
  \end{displaymath}
  and therefore \(\overline{\mu}(\fx)\neq\overline{\mu}(\fy)\) (see
  Example~\ref{p:47}(\ref{d:exs:1})). On the other hand, if \(\ftF\) has a
  separating class of monotone predicate liftings, then the Kantorovich
  extension of \(\ftF\) with respect to this class has the desired property by
  Example~\ref{p:47}(\ref{d:exs:1}).
\end{proof}

To conclude this section we note that quantale-valued lax extensions that preserve converses and satisfy the condition of Corollary~\ref{d:cor:1} lead to quantale-valued notions of bisimilarity that extend the canonical coalgebraic notion of behavioural equivalence.

\begin{proposition}
	Let \(\V\) be a non-trivial quantale,
	and let \(\eF \colon \Rels{\V} \to \Rels{\V}\) be a lax extension of a functor \(\ftF \colon \SET \to \SET\) that preserves converses and such that, for all sets \(X\), the \(\V\)-category \((\ftF X,\eF1_{X})\) is separated.
 	Then two states in an \(\ftF\)-coalgebra are behaviourally equivalent if and only if their \(\eF\)-bisimilarity is greater or equal than \(k\).
\end{proposition}
\begin{proof}
	Consider the lax homomorphisms of quantales \(\phi \colon \two \to \V\) defined by \(\phi(0) = \bot\) and \(\phi(1) = k\), and \(\psi \colon \V \to \two\) defined by \(\psi(v) = 1\) if \(k \leq v\) and \(\psi(v) = 0\), otherwise.
	It is well-known that lax homomorphisms of quantales give rive to lax functors between the corresponding categories of \(\V\)-relations \citep{HST14}.
	Hence, we obtain lax functors
	\[
		\phi \colon \REL \to \Rels{\V} \text{ and } \psi \colon \Rels{\V} \to \REL
	\]
	that act identically on sets and postcompose the lax homomorphisms of quantales with relations (interpreted as maps into the quantales).
	It is easy to see that if we start with a lax extension \(\eF \colon \Rels{\V} \to \Rels{\V}\) of \(\ftF\) that satisfies the conditions of the proposition, then, as \(\V\) is non-trivial, we obtain an identity-preserving lax extension \(\eF_2 \colon \REL \to \REL\) of \(\ftF\) that preserves converses as the composite
	\begin{displaymath}
		\begin{tikzcd}
			\Rels{\V} & \Rels{\V} \\
			\REL     & \REL.
			\ar[from=1-1, to=1-2, "\eF"]
			\ar[from=1-2, to=2-2, "\psi"]
			\ar[from=2-1, to=1-1, "\phi"]
			\ar[from=2-1, to=2-2, "\eF_2"']
		\end{tikzcd}
	\end{displaymath}
	Therefore, by \citet[Theorem~14]{MV15}, \(\eF_2\)-bisimilarity coincides with behavioural equivalence.
	Moreover, as \(\eF\)-bisimilarity is itself an \(\eF\)-bisimulation, it follows that two states in an \(\ftF\)-coalgebra are \(\eF_2\)-bisimilar if and only if their \(\eF\)-bisimilarity is greater or equal than \(k\).
\end{proof}

%
%
%
%
%
%
%
%
%
%
%
%
%

\section{Small Lax Extensions}
\label{sec:small-lax}

We next discuss the possibility of recovering a lax extension from a \emph{set} of predicate liftings.

\begin{definition}
  Let \(\lambda\) be a regular cardinal.
  A lax extension \(\eF \colon \Rels{\V} \to \Rels{\V}\) of a functor \(\ftF \colon\SET\to\SET\) is \df{\(\lambda\)-small} if it can be obtained as the Kantorovich extension of a set of \(\kappa\)-ary predicate liftings with \(\kappa<\lambda\).
 We call \(\eF\) small if it is \(\lambda\)-small for some regular cardinal \(\lambda\).
\end{definition}

We will see next that every lax extension of an accessible functor is
small. We recall that, for a regular cardinal \(\lambda\), a functor
\(\ftF\colon \SET\to \SET \) is called \df{\(\lambda\)-accessible} if \(\ftF\)
preserves \(\lambda\)-directed colimits. Furthermore, a functor
\(\ftF\colon \SET\to \SET \) is called accessible if \(\ftF\) is
\(\lambda\)-accessible for some regular cardinal \(\lambda\).

Clearly, every \(\lambda\)-accessible functor
\(\ftF\colon \SET\to \SET \) is \df{\(\lambda\)-bounded}, that is, for
every set \(X\) and every \(\fx\in \ftF X\), there exists a subset
\(m \colon A\to X\) with \(|A|<\lambda\) and \(\fx\) is in the image
of \(\ftF m\). This property is in fact equivalent to accessibility:

\begin{theorem}\citep{AMSW19}
  A functor \(\ftF \colon\SET\to\SET\) is \(\lambda\)-accessible if and only if \(\ftF\) is \(\lambda\)-bounded.
\end{theorem}
\noindent An immediate consequence of the result above is an algebraic
presentation of accessible functors.

\begin{theorem}\citep{AMMU15}
  \label{d:thm:2}
  Let \(\lambda\) be a regular cardinal and \(\ftF\colon \SET\to \SET \) be a functor.
  The following assertions are equivalent.
  \begin{tfae}
  \item\label{d:item:1} \(\ftF\) is \(\lambda\)-bounded
  \item\label{d:item:2} \(\ftF\) is a colimit of representable functors \(\hom(X,-)\) where
    \(|X|<\lambda\).
  \item There exists a small epi-cocone
    \begin{displaymath}
      {(\hom(X_{i},-) \longrightarrow\ftF)}_{i\in I}
    \end{displaymath}
    where, for each \(i\in I\), \(|X_{i}|<\lambda\).
  \end{tfae}
\end{theorem}

\begin{remark}
  The implication \ref{d:item:1}\(\implies\)\ref{d:item:2} above can be
  justified as follows. First recall \citep{Mac98} that
  \(\ftF \colon\SET\to\SET\) is a colimit of the (large) diagram
  \begin{displaymath}
    {(\text{elements \((X,\fx)\) of \(\ftF\)})}^{\op} \longrightarrow [\SET,\SET].
  \end{displaymath}
  By~\ref{d:item:1},
  \begin{displaymath}
    {(\text{elements \((X,\fx)\) of \(\ftF\) where \(|X|<\lambda\)})}^{\op} \longrightarrow
    {(\text{elements \((X,\fx)\) of \(\ftF\)})}^{\op}
  \end{displaymath}
  is cofinal.
\end{remark}

\begin{proposition}
  \label{p:100}
  Every lax extension of an \(\lambda\)-accessible functor is \(\lambda\)-small.
  In fact, \(\eF=\eF^{M_{\lambda}}\) for
  \begin{displaymath}
    M_{\lambda}=\{\text{all \(\alpha\)-ary Moss-liftings, \(\alpha<\lambda\)}\}.
  \end{displaymath}
\end{proposition}
\begin{proof}
  Let \(\eF \colon \Rels{\V} \to \Rels{\V}\) be a lax extension of an accessible
  functor \(\ftF \colon \SET \to \SET\). By Theorem~\ref{d:thm:1},
  \(\eF=\eF^{M}\) for the class \(M\) of all Moss-liftings. Let
  \(\kappa\geq\lambda\) be a cardinal and \(\fk\in\ftF\kappa\). Since \(\ftF\)
  is \(\lambda\)-accessible, there is some cardinal \(\alpha<\lambda\) with
  inclusion \(i \colon\alpha\to\kappa\) and some \(\fa\in\ftF\alpha\) with
  \(\ftF i(\fa)=\fk\). By Corollary~\ref{p:23},
  \(\eF^{\mu^{\fa}}=\eF^{\mu^{\fk}}\). Therefore \(\eF=\eF^{M_{\lambda}}\).
\end{proof}

\begin{example}
  Every lax extension of the finite powerset functor can be recovered from a countable set of predicate liftings.
\end{example}

\begin{corollary}
  Let \(\ftF\colon \SET\to \SET \) be a \(\lambda\)-accessible functor with a lax extension \(\eF\colon \Rels{\V}\to \Rels{\V}\) such that, for all sets \(X\) with \(|X|<\lambda\), the \(\V\)-category \((\ftF X,\eF1_{X})\) is separated. Then the set
  \begin{displaymath}
    M_{\lambda}=\{\text{all \(\alpha\)-ary Moss liftings, \(\alpha<\lambda\)}\}
  \end{displaymath}
  is separating.
\end{corollary}
\begin{proof}
  Same as for Corollary~\ref{d:cor:1}, using Proposition~\ref{p:100}.
\end{proof}

\begin{proposition}
  Let \(\eF_\ftF \colon \Rels{\V} \to \Rels{\V}\) be a lax extension of a functor \(\ftF \colon \SET \to \SET\) and \(i \colon \ftG\to \ftF\) be a natural transformation.
  Furthermore, let \(\eG_i \colon \Rels{\V} \to \Rels{\V}\) be the initial lax extension with respect to \(i \colon \ftG\to \ftF\).
  If \(\eF_\ftF\) is \(\lambda\)-small, then \(\eG_i\) is \(\lambda\)-small.
\end{proposition}
\begin{proof}
  Follows from Proposition~\ref{d:rem:2}.
\end{proof}

We can relax slightly the condition of Proposition~\ref{p:100} by taking advantage of the structure of \(\V\)-category.

\begin{definition}
  Let \((X,a)\) be a \(\V\)-category.
  A map \(i \colon A\to X\) is \df{dense} in \((X,a)\) if \(a = a \cdot i \cdot i^\circ \cdot a\).
\end{definition}

\begin{remark}
  A map \(i \colon A \to X\) is dense in \((X,a)\) if and only if the image of
  \(A\) is dense in \((X,a)\) with respect to the closure operator
  \(\overline{(-)}\) on \(\Cats{\V}\) introduced by \citet{HT10}. We recall that
  for every \(x \in X\) and \(M \subseteq X\),
  \begin{displaymath}
    x \in \overline{M} \iff k \le \bigvee_{z\in M} a(x,z) \otimes a(z,x).
  \end{displaymath}
  This closure operator generalizes the one induced by the usual topology
  associated with a metric space. In fact, one can show that for categories
  enriched in a \emph{value quantale}, such as metric spaces, this closure
  operator coincides with the one induced by the symmetric topology considered
  by \citet{Fla92}.
\end{remark}

\begin{definition}
  A natural transformation \(i \colon \ftG\to \ftF\) between functors
  \(\ftG,\ftF \colon\SET\to\SET\) is called \df{dense} with respect to a lax extension
  \(\eF \colon \Rels{\V} \to \Rels{\V}\) of \(\ftF\) if, for each set
  \(X\), \(i_{X}\colon \ftG X\to\ftF X\) is dense in the
  \(\V\)-category \((\ftF X,\eF 1_{X})\).
\end{definition}

The following lemma records that dense maps are compatible with Moss liftings (see Theorem~\ref{p:75}), in the sense that to determine the action of a \(\kappa\)-ary Moss lifting it suffices to focus on a dense subset of \((\ftF \kappa, \eF 1_\kappa)\).

\begin{lemma}
  \label{p:29}
  Let \(\eF \colon \Rels{\V} \to \Rels{\V}\) be a lax extension of a functor \(\ftF \colon \SET \to \SET\).
  Furthermore, let \(\varphi \colon Y\relto\kappa\) be a \(\V\)-relation, \(\fk\) be an element of \(\ftF\kappa\), and \(i \colon A \to \ftF\kappa\) be a dense map in \((\ftF\kappa,\eF 1_{\kappa})\).
  Then %
  \begin{displaymath}
    \mu^{\fk}(\varphi)= \mu^\fk(1_\kappa) \cdot i \cdot i^\circ \cdot \eF(\varphi).
  \end{displaymath}
\end{lemma}

\begin{theorem}
  Let \(i \colon \ftG\to \ftF\) be a natural transformation between functors
  \(\ftG,\ftF \colon\SET\to\SET\), and let \(\eF \colon \Rels{\V} \to \Rels{\V}\) be a lax
  extension of \(\ftF\) such that \(i\) is dense with respect to \(\eF\).
  If \(\ftG\) is \(\lambda\)-accessible, then \(\eF\) is \(\lambda\)-small.
\end{theorem}
\begin{proof}
  Let \(r \colon X \relto \kappa\) be a \(\V\)-relation, \(\fk\in\ftF\kappa\),
  and consider \(i_{\kappa} \colon \ftG \kappa \to \ftF \kappa\). Then, by
  Lemma~\ref{p:29}, for all \(r \colon X\relto Y\),
  \begin{align*}
    \eF^{\mu^\fk}(r)
    & = \bigwedge_{g \colon Y \relto \kappa} \mu^\fy (g) \multimapdot \mu^\fy (g \cdot r) \\
    & = \bigwedge_{g \colon Y \relto \kappa}
      \left(\mu^\fy(1_\kappa) \cdot i_{\kappa} \cdot i_{\kappa}^\circ \cdot \eF(g) \right)
      \multimapdot \left(\mu^\fy(1_\kappa) \cdot i_{\kappa} \cdot i_{\kappa}^\circ \cdot \eF(g \cdot r) \right).
  \end{align*}
  Hence, by Propositions~\ref{p:2}(\ref{p:14}) and~\ref{p:28},
  \begin{align*}
    \eF^{\mu^\fk}(r)
    & \geq \bigwedge_{g \colon Y \relto \kappa}
      \left( i_{\kappa}^\circ \cdot \eF(g) \right) \multimapdot
      \left(i_{\kappa}^\circ \cdot \eF(g \cdot r) \right) \\
    & = \bigwedge_{\fl \in \ftG \kappa} \bigwedge_{g \colon Y \relto\kappa}
      \mu^{i(\fl)} (g) \blackright \mu^{i(\fl)}(g \cdot r) \\
    & = \bigwedge_{\fl \in \ftG \kappa} \eF^{\mu^{i(\fl)}}(r).
  \end{align*}
  Now, the claim follows from the fact that \(\ftG\) is \(\lambda\)-accessible.
\end{proof}

Finally, we obtain a generalization of Theorem~\ref{p:60} mentioned in the introduction.

\begin{corollary}
  \label{p:33}
  Let \(\eF \colon \Rels{\V} \to \Rels{\V}\) be a lax extension.
  If there is a regular cardinal \(\lambda\) such that the natural transformation \(\ftF_\lambda \hookrightarrow \ftF\) is dense, then \(\eF\) is
  \(\lambda\)-small.
\end{corollary}

\section{Conclusions}

We have argued that the language of relations is a natural system to
express the connection between predicate liftings and lax extensions
by showing that it provides a pointfree perspective in which many
fundamental notions and results arise naturally and proofs become very
elementary.  Using this perspective, we were able to remove several
technical restrictions that feature centrally in previous approaches
which were confined to classical or $[0,1]$-valued relations and accessible
\(\SET\)-functors \citep{Lea08,MV15,WS20}.  Indeed, our
constructions and results are valid for arbitrary \(\SET\)-functors
and lax extensions to categories of quantale-enriched relations.  In
particular, we have introduced a new way of extracting predicate
liftings from a lax extension that is independent of functor
presentations, and indeed we provide an intrinsic characterization of
predicate liftings that are \emph{induced} in this sense by a given
lax extension. This leads to a very simple description of Moss
liftings, which has made it straightforward to show -- in quantalic
generality -- that every lax extension is induced by its
\emph{class} of Moss liftings, and that the role of accessibility is
to ensure that it suffices to consider a \emph{set} of Moss liftings.
Consequently, we have obtained the fact that every lax extension of a
\(\SET\)-functor is an initial extension of canonical extensions of
generalized monotone neighbourhood functors, as well as a
generalization of the fact that the finitary functors that admit an
identity-preserving lax extension are precisely the ones that admit a
separating set of monotone predicate liftings.  Furthermore, we have
lifted the result that every finitarily separable $[0,1]$-valued lax
extension of a \(\SET\)-functor is induced by a set of
predicate liftings \citep{WS20} to quantalic generality. Here, we have
avoided restrictions on the quantale that are needed when classical
notions of density~\citep{Flagg97} are used (like in recent results on
quantalic van Benthem and Hennessy-Milner theorems~\citep{WS21}), by
employing instead a categorical closure operator available on all
quantale-enriched categories.

\section*{Acknowledgements}

The first author acknowledges support by the Deutsche Forschungsgemeinschaft
(DFG, German Research Foundation) under the project \emph{A High Level Language
  for Programming and Specifying Multi-Effect Algorithms} (GO~2161/1-2, SCHR
1118/8-2). The second author acknowledges support by the Center for Research and
Development in Mathematics and Applications (CIDMA) through the Portuguese
Foundation for Science and Technology FCT -- Fundação para a Ciência e a
Tecnologia (UIDB/04106/2020). The third author acknowledges support by the DFG
-- project numbers 259234802 and %
419850228. %
The fourth author acknowledges support by the DFG -- project number
434050016. %

\bibliographystyle{apalike}
\bibliography{normal_lax}

\end{document}